\documentclass{amsart}

\usepackage[T1]{fontenc}
\usepackage[utf8x]{inputenc}
\usepackage{lmodern, euscript}
\usepackage{enumerate}
\usepackage{mathabx}
 \usepackage[all,cmtip]{xy}

\usepackage{amsfonts,amsmath,amstext,amsbsy,amssymb,amsbsy,
  amsopn,amsthm,amscd} 
\usepackage[T1]{fontenc}
\usepackage{hyperref}
\usepackage{mathrsfs}

\RequirePackage[dvipsnames]{xcolor} 
\definecolor{halfgray}{gray}{0.55} 
\definecolor{webgreen}{rgb}{0,0.5,0}
\definecolor{webbrown}{rgb}{.6,0,0} \hypersetup{%
  colorlinks=true, linktocpage=true, pdfstartpage=3,
  pdfstartview=FitV,%
  breaklinks=true, pdfpagemode=UseNone, pageanchor=true,
  pdfpagemode=UseOutlines,%
  plainpages=false, bookmarksnumbered, bookmarksopen=true,
  bookmarksopenlevel=1,%
  hypertexnames=true,
  pdfhighlight=/O,
  urlcolor=webbrown, linkcolor=RoyalBlue,
  citecolor=webgreen, 
  pdftitle={Recap},%
  pdfsubject={},%
  pdfkeywords={},%
  pdfcreator={pdfLaTeX},%
  pdfproducer={LaTeX with hyperref}%
}

\newtheorem{theorem}{Theorem}[section]
\newtheorem{lemma}[theorem]{Lemma}
\newtheorem{corollary}[theorem]{Corollary}

\theoremstyle{definition}

\newtheorem{example}[theorem]{Example}

\newtheorem{remark}[theorem]{Remark}

\newcommand{\field}[1]{\mathbb{#1}}
\newcommand{\R}{\field{R}}
\newcommand{\N}{\field{N}}

\newcommand{\Z}{\field{Z}}

\newcommand{\T}{\field{T}}

\newcommand{\K}{\field{K}}
\newcommand{\D}{\mathbb D}
\newcommand{\bS}{\mathbb S}

\newcommand{\cC}{\mathcal C}

\newcommand{\cF}{\mathcal{F}}

\newcommand{\cL}{\mathcal{L}}

\newcommand{\cP}{\mathcal{P}}

\newcommand{\cO}{\mathcal{O}}

\newcommand{\ie}{{\it i.e., } }
\newcommand{\eg}{{\it e.g., } }

\renewcommand{\phi}{\varphi}

\renewcommand{\|}{\,\Vert\,}

\begin{document}
\baselineskip=14pt

\title{Smooth rigidity for very non-algebraic expanding maps}
\author {Andrey Gogolev and Federico Rodriguez Hertz}\thanks{The authors were partially supported by NSF grants DMS-1823150 and DMS-1500947 \& DMS-1900778, respectively}

 \address{Department of Mathematics, The Ohio State University,  Columbus, OH 43210, USA}
\email{gogolyev.1@osu.edu}

\address{Department of Mathematics, The Pennsylvania State University, 
University Park, PA 16802, USA}
\email{hertz@math.psu.edu}

\begin{abstract} 
  \begin{sloppypar}
We show that the space of expanding maps contains an open and dense set where smooth conjugacy classes of expanding maps are determined by the values of the Jacobians of return maps at periodic points.
  \end{sloppypar}
\end{abstract}
\maketitle

\section{Introduction}

Let $M$ be a smooth closed manifold. Recall that a $C^r$, $r\ge 1$, map $f\colon M\to M$ is called {\it expanding } if 
$$
\|Dfv\|>\|v\|
$$
for all non-zero $v\in TM$ and some choice of Riemannian metric on $M$. It is easy to see that an expanding map is necessarily a covering map.

Recall that expanding maps have been classified up to topological conjugacy.  Shub \cite{Sh} proved that $M$ is covered by the Euclidean space and also that an expanding endomorphism of $M$ is topologically conjugate to an affine expanding endomorphism of an infranilmanifold if and only if the fundamental group $\pi_1(M)$ contains a nilpotent subgroup of finite index. Franks~\cite{Fr} showed that if $M$ admits an expanding endomorphism then $\pi_1(M)$ has polynomial growth. Finally, in 1981, Gromov \cite{Gr} completed classification by showing that any finitely generated group of polynomial growth contains a nilpotent
subgroup of finite index. Hence any expanding endomorphism is topologically conjugate to an affine expanding endomorphism of an infranilmanifold. 

Let $f_i\colon M_i\to M_i$ be $C^r$ smooth, $r\geq 1$, expanding maps $i=1,2$. Also we will assume that $f_1$ and $f_2$ are conjugated via a homeomorphism $h:M_1\to M_2$, \ie $h\circ f_1=f_2\circ h$. For example, homotopic expanding maps on the same manifold are always conjugate.

It is well known that $h$ is necessarily bi-H\"older continuous. However, a priori $h$ is not $C^1$ smooth with obvious obstructions carried by the eigendata of periodic points. That is, when $h$ is $C^1$, the differential of the return map $Df_1^n(x)$ is conjugate to $Df_2^n(h(x))$ when $x=f_1^n(x)$. A weaker necessary assumption is coincidence of {\it Jacobian data}, \ie
$$
\textup{Jac} (f_1^n)(x)=\textup{Jac} (f_2^n)(h(x))
$$
for all periodic points $x=f_1^n(x)$.

In this paper we offer the following progress for higher dimensional expanding maps.
{\it For any $r\ge 2$ there exists a $C^r$-dense and $C^1$-open subset $\mathcal U$ in the space of $C^r$ expanding maps such that if $f_1\in\mathcal U$ and $f_2$ is an expanding map which is conjugate to $f_1$ and has the same Jacobian data then the conjugacy is $C^{r-1}$.} In the proof we use the fact that $f_1$ lives on an infranilmanifold. In the next section we will give precise statements which, in particular, explicitly describe the set $\mathcal U$ in the next section. Our proof of this result was partially inspired by the Embedding theorem  (or Reconstruction theorem) of Takens~\cite{T}.

In dimension one smooth classification was already known. Indeed, Shub and Sullivan showed that for $C^r$, $r\ge 2$, expanding maps of the circle $S^1$ the above condition on coincidence of Jacobians implies that the conjugacy $h$ is $C^r$ smooth~\cite{SS}. In fact they proved a stronger result that an absolutely continuous conjugacy (which is not, a priori, even continuous) must be coincide a.e. with smooth conjugacy provided that the Jacobian of one of the expanding maps is not cohomologous to a constant. 

The analogous ``smooth conjugacy problem'' in the setting of Anosov diffeomorphisms was completely resolved by de la Llave, Marco and Moriy\'on in dimension 2~\cite{dlL0, LM, dlL}. In higher dimensions there was a lot of partial progress, \eg see~\cite{dlL2, G, KS} and references therein. However progress was made only for certain special classes of Anosov diffeomorphisms such as conformal or with a fine dominated splitting. When compared to this body of work, the current paper is very different. It relies on a fundamentally different approach --- to examine matching functions rather than matching measures. And it yields smooth classification on a large open set as opposed to characterization of smooth conjugacy classes of certain special maps.

The next section contains the statement of our main technical result Theorem~\ref{rigexpthm}. Then we state a number of corollaries for smooth conjugacy problem and discuss necessity of various assumptions. Section~3 is devoted to preliminaries on properties of the transfer operator associated to an expanding map. Section~4 contain the proof of the main theorem under an additional simplifying assumption that the underlying manifold is a torus, such an assumption makes the proof much shorter and more transparent. Then in Section~5 we prove Theorem~\ref{rigexpthm} in full generality. In Section~6 we derive all the corollaries on smooth classification problem. Then, nn Section~7, we give a number of examples of expanding maps illustrating various features of our results and proofs. Finally, in Section~8 we state a generalized factor version of Theorem~\ref{rigexpthm} and also give an application.

{\it Acknowledgement:} The second author was spending his sabbatical year in Laboratoire Paul Painlev\'e, Universit\'e de Lille during this research, he wants to thank them and specially Livio Flaminio for their warm and generous hospitality. He also thanks Livio Flaminio for all the discussions. We would like to thank Feliks Przytycki  for feedback pointing out references for Remark~\ref{remark_32}. We are grateful to Sasha Leibman and Vitaly Bergelson for their help with the elementary proof of Lemma~\ref{lemma_ratner}. Finally we would like to thank the referees for helpful feedback.

\section{The results.}

We adopt the standard convention and call a map $f\colon M\to M$ {\it $C^r$-smooth}, $r\ge 0$, if it is $\lfloor r\rfloor$ times continuously differentiable and its $C^{\lfloor r\rfloor}$-differential is H\"older continuous with exponent $r-\lfloor r\rfloor$. We also allow $r=\infty$ and $r=\omega$ (real analytic maps). One defines $C^r$ smooth functions on $M$ in a similar way. 

Recall that we denote by $f_i\colon M_i\to M_i$, $i=1,2$, $C^r$ smooth expanding maps and we assume that $f_1$ and $f_2$ are conjugated by $h$, $h\circ f_1=f_2\circ h$. 
Given functions $\phi_i:M_i\to\R$, $i=1,2$, we say that $(f_1,\phi_1)$ is {\it equivalent} to $(f_2, \phi_2)$ and write 
$$(f_1,\phi_1)\sim (f_2,\phi_2)$$
 if there exists a function $u:M_1\to\R$ such that 
$$\phi_1-\phi_2\circ h=u-u\circ f_1$$
Then, by the Livshits theorem~\cite{L}, $(f_1,\phi_1)\sim (f_2,\phi_2)$ if and only if
for every periodic point  $x\in Fix(f_1^n)$
	$$
	\sum_{k=0}^{n-1}\phi_1(f_1^k(x))=\sum_{k=0}^{n-1}\phi_2(f_2^k(h(x)))
	$$
Further, if $\phi_i$ are $C^r$ smooth then the transfer function $u$ is also $C^r$ smooth.\footnote{The Livshits Theorem for expanding maps can be proved using the standard transitive point argument~\cite{L}. There is no loss of regularity in the bootstrap argument for the transfer function, see \eg~\cite{J}.} The following is our main technical result.
	
	 

\begin{theorem}
\label{rigexpthm} 
Assume that $M_i$, $i=1,2$, are closed manifolds homeomorphic to a nilmanifold. Let $f_i\colon M_i\to M_i$, $i=1,2$, be $C^r$ smooth, $r\ge 1$, expanding maps and assume they are conjugate via a homeomorphism $h\colon M_1\to M_2$. Then there exist manifolds $\bar M_i$ (which are homeomorphic to a nilmanifold) and $C^{r}$ fibrations  $p_i:M_i\to \bar M_i$, $i=1,2$, (whose fibers are homeomorphic to a nilmanifold) and $C^{r}$ expanding maps $\bar f_i:\bar M_i\to \bar M_i$, such that $f_i$ fibers over $\bar f_i$, \ie
$$
p_i\circ f_i=\bar f_i\circ p_i,\;\;\;\;\; i=1,2
$$
The conjugacy $h$ maps fibers to fibers, \ie 
$$
p_2\circ h=\bar h\circ p_1
$$
where the induced conjugacy $\bar h:\bar M_1\to \bar M_2$, $\bar h\circ\bar f_1=\bar f_2\circ \bar h$,   is a  $C^{r}$ diffeomorphism.

Further, the fibrations $p_i$, $i=1,2$, have the following property. If $\phi_i:M_i\to\R$, $i=1,2$, are $C^r$ smooth functions such that $(f_1,\phi_1)\sim(f_2,\phi_2)$ 
  then there exist $C^{r}$ functions $\bar\phi_i:\bar M_i\to\R$, $i=1,2$, such that $\phi_i$ is cohomologous to $\bar\phi_i\circ p_i$ over $f_i$, $i=1,2$, and 
   $$
   \bar\phi_2\circ \bar h=\bar\phi_1
   $$	 
\end{theorem}
All manifolds in the above theorem, including the fibers of the fibrations, are connected. All manifolds are homeomorphic to nilmanifolds but could carry exotic smooth structure.

At this point we recommend that the reader looks at Example~6.1 to better understand the statement of the above theorem.

\begin{remark}
Manifold $\bar M_1$ may be equal to $M_1$ or may be a point or some dimension in between. In the first case we obtain that $f_1$ and $f_2$ are $C^r$ smoothly conjugate and in the second case we obtain that the functions $\phi_1$ and $\phi_2$ are cohomologous to a constant. 

Also notice that the regularity of the $f_i$'s and $\phi_i$'s may be different to start with. Then naturally one takes $r$ to be the minimal  value. Moreover, for a given pair of $f_i$, $i=1,2$, but different choices of $r\geq 1$, the resulting fibrations $p_i$, $i=1,2$, may, in fact, depend on $r$. 
\end{remark}

\begin{remark}
If one does not assume that $M_i$ are homeomorphic to a nilmanold then, instead of fibrations, the construction in the proof of Theorem~\ref{rigexpthm} yields compact foliations $\cF_i$ \ie foliations with all leaves compact. Further, by improving the argument used to show that the leaves of $\cF_i$ are compact, one can check that these foliations are generalized Seifert fibrations. The argument for compactness and the Seifert property of the foliation is independent of classification of expanding maps. Klein bottle Example~\ref{seifertfibration} shows that such foliations, indeed, can have exceptional leaves on infranilmanifolds, that is, they are not necessarily locally trivial fibrations. Hence the assumption that $M_i$ are homeomorphic to nilmanifolds is a necessary one. However, in practice, this assumption is not a big restriction. Indeed, by classification, any manifold which supports an expanding map is homeomorphic to an infranilmanifold. Hence, one can always lift given expanding maps to finite nilmanifold covers and study the problem on the cover.
\end{remark}

\begin{remark}
It will become clear from the proof of Theorem~\ref{rigexpthm} that the fibrations $p_i$ are uniquely determined by $f_i$, $h$, and $r$. However, if one does not require the latter property in the statement, \ie that ``matching'' functions $\phi_i$ are cohomologous to $\bar\phi_i\circ p_i$ then the choice of fibrations, in general, is not unique. For example, there is always the trivial fibration whose fibers are points. {
 In general there are finitely or infinitely many distinct smooth fibrations for a given expanding map and the maximal number of possible fibrations occurs when $h$ is smooth. This maximal number of fibrations is determined by the linearization of $f_i$ (see also Remark~\ref{remark62}).}   There is also a naturally defined partial order on the set of fibrations with the trivial one being subordinate to any other fibration and the one given by Theorem~\ref{rigexpthm} being the maximal one.
\end{remark}

\begin{remark} Recall that there exist expanding maps on exotic nilmanifolds, \ie manifolds homeomorphic but not diffeomorphic (or even not $PL$-homeomorphic) to nilmanifolds~\cite{FJ,FG}. Our theorem applies to such examples. Moreover, by using Gromoll's filtration and following the strategy of~\cite{FG2},  one can construct expanding map $f_1\colon M_1\to M_1$, on a nilmanifold $M_1$ and an expanding map $f_2\colon M_2\to M_2$ on an exotic nilmanifold $M_2$ in such a way that the fibrations $p_i\colon M_i\to\bar M_i$ are non-trivial, \ie $\dim M_i>\dim\bar M_i>0$. Also note that our theorem applies in the case when both $M_1$  and $M_2$ are exotic. We elaborate on this remark in Example~\ref{ex_exotic}.
\end{remark}


A linear expanding endomorphism $L$ of a $d$-dimensional torus $M$ is called {\it irreducible} if the characteristic polynomial of the integer matrix defining $L$ is irreducible over $\Z$; equivalently $L$ does not have non-trivial invariant rational subspaces. Recall that any expanding map $f\colon M\to M$ is conjugate to an expanding endomorphism $L$. We will say $f$ is {\it irreducible} if $L$ is irreducible.

\begin{corollary}
\label{cor1}
	Let $M_i$ be manifolds homeomorphic to the $d$-dimensional torus. Assume that $f_i\colon M_i\to M_i$ are $C^{r+1}$ smooth, $r\geq 1$, expanding maps. Assume that they are conjugate via $h$. Also assume that $f_1$ is irreducible and that the entropy maximizing measure for $f_1$ is not absolutely continuous with respect to Lebesgue measure. If $\textup{Jac} (f_1^n)(x)=\textup{Jac} (f_2^n)(h(x))$ for every $x\in Fix(f_1^n)$  and every $n$ then $h$ is a $C^r$ diffeomorphism.
\end{corollary}

We make four remarks pertaining this corollary.
\begin{remark}The condition on the measure of maximal entropy can be detected from a pair of periodic points. Hence the space of expanding maps which satisfy this assumption is $C^{r+1}$ dense and $C^1$ open in the space of expanding maps.
\end{remark}

\begin{remark}
The analogue of Corollary~\ref{cor1} for non-abelian nilmanifolds is vacuous. This is because every linear expanding maps on a nilmanifolds leaves invariant the fibration given by the center subgroup of the corresponding nilpotent Lie group. Indeed, the proof of Corollary~\ref{cor1} relies on absence of such fibrations (which is guaranteed by irreducibility in the toral case).
\end{remark}
\begin{remark}
\label{rem_irr}
Recall that an {\it infratorus} $M$ is a closed manifold covered by the torus $\T^d$. The Deck transformations of the covering $\T^d\to M$ have the form $x\mapsto Qx+v$ and the linear parts $Q$ form so called {\it holonomy group of $M$}.  We can define an expanding map $f\colon M\to M$ to be irreducible if its' lift to $\T^d$ is irreducible. Then Corollary~\ref{cor1} holds for such irreducible expanding maps of infratori by first passing to the torus cover and then arguing in the same way. 

However, the supply of irreducible examples of expanding endomorphisms of infratori which are not tori is rather limited. Notice that  any $Q\neq Id$ from the holonomy group has $1$ for an eigenvalue. Indeed otherwise corresponding affine map of the torus $x\mapsto Qx+v$ would have a fixed point by the Lefschetz formula. Further $L$ acts on the holonomy group by conjugation. Hence, because the holonomy group is finite, for a sufficiently large $k$, $L^k$ and $Q$ commute and, hence, $L^k$ leaves invariant the non-trivial rational subspace --- the eigenspace space of eigenvalue $1$ for $Q$. Hence all irreducible examples must become reducible after passing to a finite power. Still examples like that exist and we present one such example as Example~\ref{infratorusexample}.
\end{remark}

\begin{remark}
Define the critical regularity $r_0$ by
$$
r_0(f_1)=\min_{n\geq 1}\max_{x\in M_1}\frac{\log\|D_xf^n_1\|}{\log m(D_xf^n_1)}
$$
where $m$ is the conorm. Then by the argument of de la Llave~\cite[Section 6]{dlL} one can rectify the loss of one derivative and bootstrap the regularity of the conjugacy. That is if $r> r_0$ then the $C^{r}$ conjugacy given by Corollary~\ref{cor1} is, in fact, $C^{r+1}$. Same observation applies to other statements in this section. 

In fact $r_0(f_1)$ admits an alternative expression
$$
r_0(f_1)=\max_{p\in Per(f_1)}\frac{\lambda^+(p)}{\lambda^-(p)}
$$
 where $\lambda^{\pm}(p)$ is respectively the largest/smallest Lyapunov exponent for $f_1$ at $p$. Therefore  $r_0(f_1)$ can be computed directly from Lyapunov exponents along periodic orbits. 
To see that the two formulae give the same value $r_0(f_1)$ one can pass to the invertible solenoid diffeomorphism and apply the approximation result~\cite{WW}.

Notice also that a priori it does not follow from the hypothesis of Corollary~\ref{cor1} that $r_0(f_1)=r_0(f_2)$, however a posteriori one obtains this equality from smoothness of the conjugacy.
\end{remark}

We say that an expanding map $f\colon M\to M$ is {\it very non-algebraic} if for every $\lambda\in\Z$ and for every $m$, $1\leq m\leq \dim(M)$, there exists a periodic point $x$ of period $n$ such that $\lambda^n$ is not an eigenvalue of the $m$-fold exterior power
$$
\bigwedge^m D_xf^n
$$
Notice that this condition is open and dense.
\begin{corollary}\label{rigexpcor}
Assume that $f_i\colon M_i\to M_i$ are $C^{r+1}$ smooth, $r\ge 1$, expanding maps. Assume that they are topologically conjugate and also assume that $f_1:M_1\to M_1$ is very non-algebraic. Furthermore, assume that for every periodic point $x$ of $f_1$ of period $n$ 
$$\textup{Jac} (f_1^n)(x)=\textup{Jac} (f_2^n)(h(x))$$ 
Then $h$ is a $C^r$ diffeomorphism.
\end{corollary}

\begin{remark} It will be clear from the proof that the very non-algebraic assumption can be weakened to asking that for $m=1,2,\ldots, \dim(M)$ if $\lambda\in\Z$ appears  in the spectrum of 
$$
\bigwedge^m Df_{1*}
$$
then $\lambda^n$ does not appear in the spectrum of
$$
\bigwedge^m D_xf_1^n
$$
for some periodic point $x$, $x=f_1^nx$. Here $f_{1*}$ stands for the linear expanding automorphism induced by $f_1$ on the nilpotent Lie group and $Df_{1*}$ is the corresponding Lie algebra automorphism.

Note that the very non-algebraic condition prevents $f_1$ from being linear.
\end{remark}


Given two linear maps $D_i\colon\R^d\to\R^d$, $i=1,2$, we say that $D_1$ and $D_2$ have {\it disjoint spectrum} if for every $m=1,\ldots d$, the $m$-th exterior powers $\wedge^m D_1$ and $\wedge^m D_2$ do not share any real eigenvalues. Given two periodic points $x=f^k(x)$ and $y=f^l(y)$ we say that they have {\it disjoint spectrum} if the differentials $D_xf^{kl}$ and $D_yf^{kl}$ have disjoint spectrum.

\begin{corollary}\label{rigcor2points}
 Assume that $f_i\colon M_i\to M_i$ are $C^{r+1}$ smooth, $r\ge 1$, expanding maps. Assume that they are conjugate and also assume that there exists $f_1$-periodic points $x$ and $y$ which have disjoint spectrum. If for every periodic point $x$ of $f_1$ of period $n$ the Jacobians $\textup{Jac}(f_1^n)(x)$ and $\textup{Jac}(f_2^n)(h(x))$ coincide then $f_1$ is $C^{r}$ conjugate to $f_2$.
\end{corollary}
Corollary \ref{rigcor2points} follows directly from Corollary \ref{rigexpcor} since the property of having two periodic points with disjoint spectrum directly implies the very non-algebraic property.

Recall that a homeomorphism is called absolutely continuous if it send the Lebesgue measure to a measure which is absolutely continuous with respect to Lebesgue measure.

\begin{corollary}\label{shubsullcor}
	Let $r\ge 1$. If two $C^{r+1}$ very non-algebraic expanding maps which are conjugate via an absolutely continuous homeomorphism $h$ then $h$ is, in fact, $C^r$ smooth.
	 \end{corollary}

Corollary \ref{shubsullcor} follows directly from Corollary \ref{rigexpcor}. Indeed, by ergodicity $h$ must map the smooth absolutely continuous measure of $f_1$ to the smooth absolutely continuous measure for $f_2$. It follows that the Jacobians at corresponding periodic points must be equal.


\section{Krzy\.{z}ewski-Sacksteder Theorem for expanding maps}\label{sectionKS}

Given a $C^r$, $r\ge 1$, expanding map $f\colon M\to M$ and a $C^r$ potential $\phi\colon M\to \R$ the {\it transfer operator} $\cL_{\phi,f}\colon C^{k}(M)\to C^{k}(M)$ given by
$$
\cL_{\phi,f}u(x)=\sum_{y\in f^{-1}x}e^{\phi(y)}u(y)
$$
is defined for $C^{k}$ functions $u$, where $k\le r$.
When no confusion is possible we abbreviate the notation for the transfer operator to $\cL_\phi$.

\begin{theorem}[Ruelle-Perron-Frobenius/Krzy\.{z}ewski-Sacksteder]
	\label{thm_SK}
	Let $f\colon M\to M$ be a $C^r$, $r\geq 1$, expanding map and let $\phi\colon M\to \R$ be a $C^r$ potential; let $0\le k\le r$. Then the transfer operator $\cL_\phi\colon C^k(M)\to C^k(M)$ has a unique maximal positive eigenvalue $e^c$
	$$
	\cL_\phi e^u=e^{c+u}
	$$
	Corresponding eigenfunction $e^u$ is positive and is unique up to scaling. The eigenvalue $e^c$ and the eigenvalue $e^u$ are independent of the choice of $k\in[0, r]$. Further, $e^u$ is $C^{r}$ smooth.
\end{theorem}

\begin{remark} \label{remark_32} Originally  this theorem was established by Ruelle for a more general class of expanding maps and in H\"older regularity~\cite{R1,R2} (see also~\cite[1.7]{B}). Sacksteder~\cite{S} and Krzy\.{z}ewski~\cite{Krz} had independently established regularity of the eigenfunction. Krzy\.{z}ewski~\cite{Krz2} has done the analytic case as well. We note that both Sacksteder and Krzy\.{z}ewski only considered the case when $\phi=-\log \textup{Jac}(f)$ because they were interested in regularity of the smooth invariant measure for $f$. However the proofs work equally well for arbitrary smooth potentials. Note that the uniqueness of the eigenspace occurs already  among continuous functions provided that the potential is at least H\"older.

	Another comment is that when $r$ is an integer the proof of Sacksteder only yields $(r-1)+Lip$ regularity of the eigenfunction $e^u$. The $C^r$ regularity of $e^u$ was established by Szewc~\cite{Sz}, see also~\cite[Theorem 8.6.3]{BG} for an exposition in the one-dimensional case.
\end{remark}

\begin{corollary} 
	\label{cor}
	Let $f$ and $\phi$ be the same as in Theorem~\ref{thm_SK}.
	Then there exists a unique $C^{r}$ smooth function $\hat \phi\colon M\to \R$ and a unique constant $c$ given by Theorem~\ref{thm_SK} such that 
	\begin{enumerate}
		\item $\hat \phi+c$ is cohomologous to $\phi$; 
		\item 1 is the maximal eigenvalue of the transfer operator $\cL_{\hat\phi}$;
		\item $\cL_{\hat\phi}1=1$
	\end{enumerate}
\end{corollary}

\begin{proof}
	Let $e^c$ be the maximal eigenvalue with eigenfunction $e^u$ for $\cL_\phi$ given by Theorem~\ref{thm_SK}  
	$$\cL_\phi e^u=e^{c+u}$$ 
	Let 
	$\hat\phi=\phi-c+u-u\circ f$
	then 
	$$\cL_{\hat\phi}1=1$$
	It is also clear that $1$ is the maximal eigenvalue of $\cL_{\hat\phi}$ since otherwise $e^c$ would not be maximal positive eigenvalue for $\cL_\phi$. 
	
	Further, assume that $c'\in\R$ and $\hat\phi'$ continuous also satisfy the conclusion of the corollary with
	$$\hat\phi'=\phi-c'+u'-u'\circ f$$
	Then by the same calculation we have
	$$
	\cL_\phi e^{u'}=e^{c'+u'}
	$$
	with $e^{c'}$ being the maximal positive eigenvalue. Hence, by the uniqueness part of Theorem~\ref{thm_SK} we obtain that $c=c'$ and $u=u'$.
\end{proof}

Such normalized potentials $\hat\phi$ have been recently studied in the context of thermodynamical formalism~\cite{GKLM}.

\begin{remark}
\label{rem_c}
Constant $c$ equals to topological pressure $P(\phi)$. It follows that if $(f_1,\phi_1)\sim (f_2,\phi_2)$ then the maximal eigenvalue is the same for corresponding operators and hence $(f_1,\hat\phi_1)\sim(f_2,\hat \phi_2)$. (But we won't use this fact.)
\end{remark}

\begin{remark}
	Let $e^c$ be the maximal positive eigenvalue for $\cL_\phi$ with eigenfunction $e^u$ and assume that $e^w$ is another positive continuous eigenfunction for $\cL_\phi$, \ie $\cL_{\phi}e^w=\sigma e^w$ for some $\sigma\in\R$, then $w=u+k$ for some $k\in\R$ and $\sigma$ is the maximal eigenvalue. Notice that it follows that condition $2$ of Corollary~\ref{cor} is automatic from condition $3$ because a positive eigenfunction necessarily corresponds to the maximal eigenvalue. (But we won't use this fact.)
\end{remark}



\section{Proof of the Main Theorem: the torus case}
\label{sec_fib}

The proof of Theorem~\ref{rigexpthm} consists of two steps. The first step is to built the fibrations and the second step is to verify the posited property of the fibrations. In this section will prove Theorem~\ref{rigexpthm} under an {\it additional assumption that $M_i$ are homeomorphic to a torus}. This assumption simplifies quite a bit the construction of fibrations. The second step is general and does not rely on homotopy type of $M_i$. Building fibrations in the case when $M_i$ are general nilmanifold requires a more complicated argument that involves and inductive procedure on the degree of nilpotency of the the fundamental group. This more general argument appears in Section~\ref{sec_nilmanifold}.

\subsection{Fibrations}

We begin by explaining the construction of fibrations $p_i$, $i=1,2$, which appear in Theorem~\ref{rigexpthm}. 

Recall that $h\circ f_1=f_2\circ h$ and consider the following space of pairs of smooth functions
$$
V=\{(\psi_1,\psi_2)\in C^{r}(M_1)\times C^{r}(M_2): \psi_1=\psi_2\circ h\}
$$
This is a closed subspace of $C^{r}(M_1)\times C^{r}(M_2)$. Note that if $(\psi_1,\psi_2)\in V$ then $(\psi_1\circ f_1,\psi_2\circ f_2)\in V$. Also note that $V$ always contains constants $(c,c)$ and is an algebra. We denote by $V_i$ the projection of $V$ on $C^{r}(M_i)$, $i=1,2$. 


%
%
%
%
%

Define the subspace fields $E_i(x)\subset T_xM_i$, $i=1,2$, 
$$
E_i(x)=\bigcap_{\psi_i\in V_i^r}\ker d_x\psi_i\,\,\,\,
$$
Notice that if $x_n\to x$, $n\to\infty$, then $\limsup E_i(x_n)\subset E_i(x)$. 
This property implies that the dimension function $\dim E_i(x)$ is upper semicontinuous. 
Let $m_i=\min_{x\in M_i} \dim E_i(x)$, then upper semicontinuity implies that the set 
$$
U_i=\{x\in M_i: k_i(x)=m_i\}
$$
 is open. 
 \begin{lemma} $U_i=M_i$, $i=1,2$ and hence $E_i$, in fact, is a distribution.
 \label{new_lemma}
 \end{lemma}
 
 \begin{proof}
 Let $\Gamma^n_i$ be the group of Deck transformations of the covering map $f_i^n\colon M_i\to M_i$, $i=1,2$. Deck transformations are $C^r$ diffeomorphisms.
 By definition
 $$
 \Gamma^n_i=\{T: \, f_i^n\circ T=f_i^n\}
 $$
 and, hence, $h_\#\colon\Gamma^n_1\to\Gamma^n_2$ given by $h_\# (T)=h\circ T\circ h^{-1}$ is an isomorphism. Indeed, if $T\in \Gamma^n_1$ then $f_2^n\circ (h\circ T\circ h^{-1})=h\circ f_1^n\circ T\circ h^{-1}=h\circ f_1^n\circ h^{-1}=f_2^n$ and vice versa. 
 
 Now it is easy to see that $V_i$ are $\Gamma^n_i$-invariant, that is, if $(\psi_1,\psi_2)\in V$ then $(\psi_1\circ T,\psi_2\circ h_\#(T))\in V$ for all $T\in \Gamma^n_1$. Indeed,
 $$
 \psi_2\circ h_\#(T) \circ h =\psi_2\circ h \circ T =\psi_1\circ T 
 $$
 Hence, for all $T\in \Gamma^n_i$ we have
 $$
 E_i(T(x))=\bigcap_{\psi\in V_i}\ker d_{T(x)}\psi=\bigcap_{\psi\in V_i}\ker d_{T(x)}(\psi\circ T)=\bigcap_{\psi\in V_i}DT(\ker d_{x}(\psi))=DT(E_i(x))
 $$
 Hence $E_i$ is $\Gamma^n_i$-invariant and, in particular, the set $U_i$ is $\Gamma^n_i$-invariant.
 
 Because $\pi_1M_i=\Z^d$ is abelian the covering $f_i^n$ is normal and $\Gamma^n_i(x)=f_i^{-n}(x)$. Hence the orbits $\Gamma^n_i(x)$ become arbitrarily dense as $n\to\infty$ and because $U_i$ is open we will have that for a sufficiently large $n$ we have
 $
 U_i=\Gamma_i^n(U_i)=M_i.
 $
 \end{proof}

It is easy to see now that the distributions $E_i$ integrate to $C^r$ foliation $\cF_i$.
Indeed, for every $x\in M_i$ there exist finitely many functions $\psi_i^1,\dots \psi_i^{d-m_i}\in V_i$ such that
$$
E_i(x)=\bigcap_{j=1}^{d-m_i}\ker d_x\psi_i^j
$$
Indeed, just take $\psi_i^j$ such that $\{d_x\psi_i^j\}_j$ is a maximal linearly independent set of $\{d_x\psi\}_{\psi\in V_i}$.

By continuity of $d\psi_i^j$ and since $E_i$ has constant dimension, the same formula holds on a small neighborhood of $x$. That is, there exists a neighborhood  $U_{i,x}$ of $x$ such that
$$
E_i(y)=\bigcap_{j=1}^{d-m_i}\ker d_y\psi_i^j
$$
for all $y\in U_{i,x}$. Therefore, by the implicit function theorem, we have that the maps  $\Psi_{i,x}:U_{i,x}\to\R^{d-m_i}$, 
\begin{equation*}\label{Psi}
\Psi_{i,x}(y)=(\psi_i^1(y),\dots \psi_i^{d-m_i}(y))
\end{equation*}
define a foliation atlas of a $C^{r}$ foliation which is tangent to $E_i$. We denote these foliations by $\cF_i$, $i=1,2$. 

\begin{lemma}\label{constantonleave}
	The leaves of $\cF_i$ are compact. In fact, the leaf $\cF_i(x)$ for $x\in M_i$, is the connected component of $x$ of the intersection 
	$$\bigcap_{\psi\in V_i}\psi^{-1}(\psi(x))$$
\end{lemma}
\begin{proof}
	Let $\psi$ be a function in $V_i$ and let $x\in M_i$. Then by the definition 
	$$T\cF_i(y)=E_i(y)\subset \ker d_y\psi$$ 
	for every $y\in\cF_i(x)$. 
	Hence $\psi$ is constant on $\cF_i(x)$ and $\cF_i(x)\subset \psi^{-1}(\psi(x))$. Hence
	$$\cF_i(x)\subset\bigcap_{\psi\in V^r_i}\psi^{-1}(\psi(x))
	$$
	 On the other hand, recall that, locally, for sufficiently small neighborhood $U_{i,x}\ni x$ we have the foliation chart and hence
	 $$
	 \cF_i(x)\cap U_{i,x}=\Psi_{i,x}^{-1}(\Psi_{i,x}(x))=\bigcap_{j=1}^{d-m_i}(\psi_i^{j})^{-1}(\psi_i^j(x))\cap U_{i,x}\supset \bigcap_{\psi\in V^r_i}\psi^{-1}(\psi(x))\cap U_{i,x}
	 $$ 
and the main claim of the lemma follows. 
\end{proof}

Recall that for every function $\psi_1\in V^r_1$ there is $\psi_2\in V^r_2$ such that  $\psi_2\circ h=\psi_1$ and vice versa. This implies that $h(\cF_1(x))=\cF_2(h(x))$ for every $x\in M_1$. Hence by the invariance of domain theorem we obtain $m_1=m_2$, \ie the dimensions of foliations $\cF_1$ and $\cF_2$ are the same. Also note that Lemma~\ref{constantonleave} and $V_i\circ f_i\subset V_i$ immediately implies that $\cF_i$ is invariant under $f_i$.

To conclude that compact $C^r$ foliations $\cF_i$ are, in fact, fibrations we need to rely on global structural stability of expanding maps and complete the argument on the ``linear side." Namely, we have that $h_i\circ f\circ h_i^{-1}=A\colon\T^d\to\T^d$ is an expanding endomorphism. Then $F=h_i(\cF_i)$ is an $A$-invariant compact continuous foliation on $\T^d$. The action of $\Gamma_i^n$ is conjugate via $h_i$ to the translation action by the set $A^{-n}(\{ x_0\})$, where $x_0$ is a fixed point of $A$ which we identify with $0\in\T^d$. Because the set $\cup_{n\ge 0} A^{-n}(\{ x_0\})$ is dense in $\T^d$ we conclude that $F$ is invariant under the $\T^d$-action on itself by translations. Hence, $\forall y\in F(x_0)$ we have $y+F(x_0)=F(y+x_0)=F(y)=F(x_0)$ and 
$F(-y)=-y+F(x_0)=-y+F(y)=F(-y+y)=F(x_0)$; that is, $F(x_0)$ is a subgroup of $\T^d$. Also recall that $F(x_0)$ is compact and connected. Hence, one can easily check (or use Cartan's closed subgroup thereom) that $F(x_0)$ is a linearly embedded subtorus $\T^m\subset\T^d$. And because $F$ invariant under translations, we conclude that $F$ is a linear fibration $\T^m\to\T^d\to\T^{d-m}$. It remains to recall that $\cF_i=h_i^{-1}(F)$ and, therefore $\cF_i$ is a fibration whose fiber is homeomorphic to $\T^m$ and whose base $\bar M_i$ is a $C^r$ manifold homeomorphic to $\T^{d-m}$.


Because $h$ sends $\cF_1$ to $\cF_2$ it induces a homeomorphism $\bar h\colon \bar M_1\to\bar M_2$. To see that $\bar h$ is smooth consider foliations charts around $x$ and $h(x)$, $x\in M_1$, given by
$$
\Psi_{1,x}(y)=(\psi_{1,x}^1(y),\dots \psi_{1,x}^{d-m_1}(y)),\,\,\,\,\mbox{and}\,\,\,\, \Psi_{2,h(x)}(y)=(\psi_{2,h(x)}^1(y),\dots \psi_{2,h(x)}^{d-m_2}(y))
$$
respectively. In these local coordinates $\bar h$ is given by $\bar h(\Psi_{1,x}(y))=\Psi_{2,h(x)}(h(y))$. However, by defintion, we know that there exist $C^r$ functions $\psi_{1,h(x)}^j$ which satisfy $\psi_{1,h(x)}^j=\psi_{2,h(x)}^j\circ h$, $j=1,\dots d-m_2$. Hence, $\bar h$ is given by
$$
\bar h(\psi_{1,x}^1(y),\dots \psi_{1,x}^{d-m_1}(y))=(\psi_{1,h(x)}^1(y),\dots \psi_{1,h(x)}^{d-m_2}(y))
$$
and since $\Psi_{1,x}$ is a $C^r$ submersion we conclude that $\bar h$ is $C^r$ on a neighborhood of $p_1(x)$. A symmetric argument proves that $\bar h^{-1}$ is $C^r$.

\subsection{Second step of the proof of Theorem~\ref{rigexpthm}: verifying the matching property}
Finally we need  to show that given $(f_1,\phi_1)\sim (f_2,\phi_2)$ we have that $\phi_i$ are cohomologous to functions in $V_i^r$. 




By Corollary \ref{cor} we have $C^r$ functions $\hat\phi_i$ and constants $c_i\in\R$ such that $\phi_i$ is $f_i$-cohomologous to $\hat\phi_i+c_i$ and we also have $\cL_{\hat\phi_1,f_1} 1=1$, $\cL_{\hat \phi_2,f_2} 1=1$. Moreover, $\hat\phi_i$ are unique among the functions cohomologous to $\phi_i$ up to a constant with this property. We know that $\hat\phi_2\circ h$ is cohomologous to $\hat\phi_1+c_2-c_1$. In fact, we will show that 
$$\hat\phi_2\circ h=\hat\phi_1$$  

By direct calculation, we have that $$(\cL_{\hat \phi_2,f_2}v)\circ h=\cL_{\hat \phi_2\circ h,f_1}(v\circ h)$$ for every function $v$. In particular, for the constant function $v=1$ we have
 $$
 1=1\circ h=(\cL_{\hat \phi_2,f_2}1)\circ h=\cL_{\hat \phi_2\circ h,f_1}(1\circ h)=\cL_{\hat \phi_2\circ h,f_1}(1)
 $$
Since $\hat\phi_2\circ h$ is cohomologous to $\phi_1$ up to a constant we get that $\hat\phi_2\circ h=\hat\phi_1$.  Hence $(\hat \phi_1,\hat \phi_2)\in V^r$ and, by the definition of foliations $\cF_i$, we conclude that $\hat \phi_i$ is constant on $\cF_i$, $i=1,2$.
It remains to set $\bar\phi_i(p_i(x))=\hat\phi_i(x)+c_i$.

\section{Proof of the Main Theorem: building fibrations on nilmanifolds}

\label{sec_nilmanifold}

In this section we built the fibrations in the general case when $M_i$ are homeomorphic to a nilmanifold $N/\Gamma$. Recall that, by classification, there is an expanding automorphism $A\colon N\to N$, $A(\Gamma)\subset\Gamma$, which induces an algebraic expanding map $N/\Gamma\to N/\Gamma$ topologically conjugate to $f_i\colon M_i\to M_i$, $i=1,2.$ The rest of the proof of Theorem~\ref{rigexpthm}, that is, verification of the matching property of fibrations, was already done in the second half of Section~\ref{sec_fib}.

Define the subspace fields $E_i(x)\subset T_xM_i$, $i=1,2$, and level sets as follows
$$
E_i(x)=\bigcap_{\psi\in V_i}\ker d_x\psi,\,\,\,\,
\cP_i(x)=\bigcap_{\psi\in V_i}\psi^{-1}(\psi(x))
$$
and let  $\cF_i(x)=cc_x \cP_i(x)$, where $cc_x$ stands for the ``connected component of $x$." Our goal is to show that $\cF_i$ are, in fact, $C^r$ fibrations with fiber and base both homeomorphic to nilmanifolds.

\begin{remark}  If $\dim E_1(x)=0$ at one point $x\in M_1$ then it is easy to conclude using the inverse function theorem that the conjugacy $h$ is $C^r$ on a neighborhood of $x$ and then, using dynamics, that $h$ is $C^r$ globally. Thus the main interest of the proof to follow is in the case when $\dim E_i\ge 1$.
\end{remark}

\subsection{Algebraic lemmas} Recall that $\Gamma$ is a lattice in a simply connected nilpotent Lie group $N$ and, hence, $\Gamma$ is torsion free and nilpotent. Let $\Gamma^0=\Gamma$ let $\Gamma^j=[\Gamma,\Gamma^{j-1}]$ be the lower central series. Denote by $k$ the smallest number such that $\Gamma^{k+1}=\{0\}$. Recall that $A(\Gamma)\subset \Gamma$ and, hence, we also have $A(\Gamma^j)\subset \Gamma^j$. Now define the following lattice
$$
A^*\Gamma^j=A^{-1}(\Gamma^j)\cdot\Gamma
$$
Note that $A^*\Gamma^0=A^{-1}(\Gamma)$ and $A^*\Gamma^{k+1}=\Gamma$. The following lemma implies that $A^*\Gamma^j$ is indeed a well-defined group.
\begin{lemma}
\label{lemma_commute}
$ A^{-1}(\Gamma^j)\cdot\Gamma=\Gamma\cdot A^{-1}(\Gamma^j)$, $j=0, \ldots k+1$.
\end{lemma}
\begin{proof} Let $\alpha\in\Gamma^j$ and $\gamma\in\Gamma$. Then
$$
A^{-1}(\alpha)\gamma=A^{-1}(\alpha A(\gamma))=A^{-1}(A(\gamma)\alpha c)
$$
where $c$ is a commutator, $c\in[\Gamma, \Gamma^j]=\Gamma^{j+1}\subset\Gamma^j$. Hence 
$A^{-1}(\alpha)\gamma=\gamma A^{-1}(\alpha c)\in\Gamma\cdot A^{-1}(\Gamma^j)$. This proves the inclusion $ A^{-1}(\Gamma^j)\cdot\Gamma\subset\Gamma\cdot A^{-1}(\Gamma^j)$ and the reverse inclusion follows from a similar  calculation.
\end{proof}

\begin{lemma}
\label{lemma_normal}
The group $A^*\Gamma^{j+1}$ is a normal subgroup of $A^*\Gamma^j$, $j=0, \ldots k$.
\end{lemma}
\begin{proof} We will use a group element $\alpha\gamma\in A^{-1}\Gamma^j\cdot \Gamma$ to conjugate an element $\beta\delta\in A^{-1}\Gamma^{j+1}\cdot \Gamma$ and we will see that the result is in $A^*\Gamma^{j+1}$.
$$
\alpha\gamma\beta\delta\gamma^{-1}\alpha^{-1}=(\alpha\gamma\alpha^{-1}\gamma^{-1})\gamma(\alpha\beta\alpha^{-1}\beta^{-1})\beta(\alpha\delta\gamma^{-1}\alpha^{-1}\gamma\delta^{-1})\gamma^{-1}\delta
$$
Indeed we have written as a product of elements in $A^*\Gamma^{j+1}$ and commutators from $[A^{-1}(\Gamma^j),\Gamma]=A^{-1}[\Gamma^j,A(\Gamma)]\subset A^{-1}[\Gamma^j,\Gamma]=A^{-1}(\Gamma^{j+1})\subset A^*\Gamma^{j+1}$.
\end{proof}

We also recall that $\Gamma^j\subset A^{-1}\Gamma^j$ and $\Gamma\subset A^*\Gamma^j$ are finite index subgroups.

\subsection{The setup on universal covers} The expanding maps $f_i$ are conjugate to the algebraic expanding map via conjugacies $h_i\colon M_i\to N/\Gamma$. Let $x_i=h_i^{-1}(id_N\Gamma)$ and let $\pi_i\colon (\tilde M_i,\tilde x_i)\to (M_i, x_i)$ be the universal covers, $i=1,2$. We denote by $\Gamma_i=\{T: \pi_i\circ T=\pi_i \}\simeq\Gamma$ the group of Deck transformation of $\pi_i$, which we can also identify with the fundamental group $\pi_1(M_i,x_i)$. Next we lift $h_i$ and $f_i$ to the universal covers in such a way that $\tilde h_i(\tilde x_i)=id_N$ and $\tilde f_i(\tilde x_i)=\tilde x_i$. Then we have
$$
\tilde h_i\circ \tilde f_i=A\circ \tilde h_i,\,\,\, \tilde h\circ \tilde f_1 = \tilde f_2\circ \tilde h,
$$
where $\tilde h=\tilde h_2^{-1}\circ \tilde h_1\colon \tilde M_1\to \tilde M_2$. We also have $\tilde f_i\circ T=A(T)\circ\tilde f_i$, for $T\in\Gamma_i$.

The group $A^{-1}(\Gamma)$ acts on $N$ by left translations $x\mapsto A^{-1}(T)\cdot x=(A^{-1}\circ T\circ A)(x)$, $T\in\Gamma$.
Following the same idea as in~Section~\ref{sec_fib} we can conjugate this action using $\tilde h_i$ and obtain actions on $\tilde M_i$
$$
A^{-1}\Gamma_i=\{\tilde f_i^{-1}\circ T\circ \tilde f_i: T\in \Gamma_i\}
$$ 
Furthermore we can similarly consider the following actions for any $j=0,\ldots k+1$
$$
A^{-1}\Gamma^j_i=\{\tilde f_i^{-1}\circ T\circ \tilde f_i: T\in \Gamma_i^j\},\,\,$$ $$ A^*\Gamma^j_i=A^{-1}\Gamma^j_i\circ \Gamma_i=\{f_i^{-1}\circ T\circ f_i \circ S: T\in \Gamma_i^j, S\in \Gamma_i\}
$$
Clearly the actions of $A^*\Gamma^j_1$ and $A^*\Gamma^j_2$ are conjugate via $\tilde h$. Consider the orbits
$$
\tilde \cO_i^j(x)=A^*\Gamma^j_i(x), x\in\tilde M_i
$$
It is immediate that the orbits are $\Gamma_i$-invariant: $\tilde \cO_i^j(T(x))=\tilde \cO_i^j(x)$, $T\in\Gamma_i$. Hence the projection  of the orbit $\cO_i^j(\pi_i(x))=\pi_i(\tilde \cO_i^j(x))$ is a finite set of cardinality $|A^*\Gamma^j/\Gamma|$. Further these partitions into orbits are invariant under the expanding maps: $\tilde f_i(\tilde \cO_i^j(x))\subset \tilde \cO_i^j(\tilde f_i(x))$, $ f_i( \cO_i^j(x))\subset  \cO_i^j(f_i(x))$. Indeed, let $T\in  \Gamma_i^j$ and $S\in\Gamma_i$ then
$$
\tilde f_i((\tilde f_i^{-1}\circ T\circ \tilde f_i\circ S) (x))=(T\circ \tilde f_i\circ S)(x)=(T\circ A(S))(\tilde f_i(x))\in\tilde \cO_i^j(\tilde f_i(x))
$$
because $T\in\Gamma_i^j\subset A^{-1}\Gamma_i^j$.

Note that $\tilde \cO_i^{k+1}(x)$ is just that $\Gamma_i$-orbit of $x$ and, hence, $ \cO_i^{k+1}(x)=\{x\}$; and $\tilde \cO_i^{0}(x)$ is $A^{-1}\Gamma_i$ orbit of $x$ and, hence, $\cO_i^0(x)=f_i^{-1}(f_i(x))$, while $ \cO_i^{j}(x)$, $j=1,\ldots k$ interpolate in between. 

\begin{remark} The fact that $\cO_i^0(\pi_i(x))$ are not orbits of a finite Deck group action on $M_i$ is forcing us to work on the universal cover; cf.~Section~\ref{sec_fib}.
\end{remark}

We now make an observation, which will be very important in the sequel, that we can consider the same setting for expanding maps $f_1^n$, $f_2^n$ and the expanding endomorphism $A^n$ for any $n\ge 1$, making the above discussion the case when $n=1$. Namely, we have actions of $A^{-n}\Gamma_i^j$ and of $A^{*n}\Gamma_i^j=A^{-n}\Gamma_i^j\cdot \Gamma_i$,  on the universal covers $\tilde M_i$, $i=1,2$, which are conjugate via $\tilde h$. Also, the action of $A^{*n}\Gamma_i^0=A^{-n}\Gamma_i$ is conjugate via $\tilde h_i$ to the action by left translations by elements of $A^{-n}\Gamma$ on $N$. Hence we can consider the group
$$
A^{-\infty}\Gamma=\bigcup_{n\ge 1} A^{-n}\Gamma
$$
which is a dense subgroup of $N$, and its actions $A^{-\infty}\Gamma_i$ on $\tilde M_i$, $i=1,2$.

\subsection{Invariance of the level set partition} To set up an induction argument we introduce ``interpolating" subspace fields and level sets as follows. Recall that 
$$
V=\{(\psi_1,\psi_2)\in C^{r}(M_1)\times C^{r}(M_2): \psi_1=\psi_2\circ h\}
$$
Denote by $\tilde V$ the corresponding space of lifted pairs 
$$
\tilde V=\{(\psi_1,\psi_2)\in C^{r}(\tilde M_1)\times C^{r}(\tilde M_2): \psi_i\circ T=\psi_i \,\,\,\, \forall T\in\Gamma_i, i=1,2; \,\, \psi_1=\psi_2\circ \tilde h\}
$$
and denote by $\tilde V_i$ the projection of $V$ on $i$-th coordinate. 

Now consider the filtration $\tilde V^0\subset \tilde V^1\subset\ldots \subset \tilde V^{k+1}=\tilde V$ given by
$$
\tilde V^j=\{(\psi_1,\psi_2)\in C^{r}(\tilde M_1)\times C^{r}(\tilde M_2):  \psi_i\circ T=\psi_i \,\,\,\, \forall T\in A^*\Gamma^j_i, i=1,2; \,\, \psi_1=\psi_2\circ \tilde h\}
$$
As before, we will use $\tilde V_i^j$ to denote the projection of $\tilde V^j$ on $i$-th coordinate.
Let $V^j=(\pi_1^{-1},\pi_2^{-1})\circ \tilde V^j$ which is well defined due to equivarience. Hence we have corresponding filtration on $M_i$ --- $ V_i^0\subset  V_i^1\subset\ldots \subset  V_i^{k+1}= V_i$. Note that functions on $V^j_i$ are precisely those functions from $V_i$ which are constant on $\cO_i^j(x)$, $x\in M_i$.

Define
$$
\tilde E_i^j(x)=\bigcap_{\psi\in \tilde V_i^j}\ker d_x\psi,\,\,\,\,
\tilde \cP_i^j(x)=\bigcap_{\psi\in \tilde V_i^j}\psi^{-1}(\psi(x)),\,\,\,\, \tilde \cF_i^j(x)=cc_x \tilde \cP_i^j(x)
$$
In the same way define 
$$
E_i^j(x)=\bigcap_{\psi\in  V_i^j}\ker d_x\psi,\,\,\,\,
 \cP_i^j(x)=\bigcap_{\psi\in  V_i^j}\psi^{-1}(\psi(x)),\,\,\,\, \cF_i^j(x)=cc_x  \cP_i^j(x)
$$
Also given a set $O$ define 
$$\tilde \cP_i^j(O)=\bigcup_{x\in O}\tilde \cP_i^j(x)$$
and similarly define sets $ \cP_i^j(O)$.

Immediately from definitions we have the following properties:
\begin{enumerate}
\item $E_i^{k+1}=E_i$, $\cP_i^{k+1}=\cP_i$ and $\cF_i^{k+1}=\cF_i$;
\item $\cP_i^j$ and $\cF_i^j$ are well-defined partitions of $M_i$ and $\tilde \cP_i^j$ and $\tilde \cF_i^j$ are well-defined partitions of $\tilde M_i$;
\item $\tilde E_i^j$, $\tilde \cP_i^j$ and $\tilde \cF_i^j$ are $A^*\Gamma_i^j$-invariant;
\item $Df_i(E_i^j(x))\subset E_i^j(f_i(x))$, $f_i(\cP_i^j(x))\subset \cP_i^j(f_i(x))$ and $f_i(\cF_i^j(x))\subset \cF_i^j(f_i(x))$, $x\in M_i$; and similarly for $\tilde E_i^j$, $\tilde \cP_i^j$ and $\tilde \cF_i^j$;
\item $h(\cP^j_1(x))=\cP^j_2(h(x))$ and $h(\cF^j_1(x))=\cF^j_2(h(x))$; and similarly for $\tilde \cP_i^j$ and $\tilde \cF_i^j$;
\item $E_i(x)=E_i^{k+1}\subset E_i^k(x)\subset \ldots \subset E_i^0(x), x\in M_i$; and similarly for $\tilde E_i^j$;
\item$\cP_i(x)=\cP^{k+1}_i(x) \subset \cP_i^k(x)\subset \ldots \subset \cP_i^0(x), x\in M_i$; and similarly for $\tilde\cP_i^j$, $\cF_i^j$ and $\tilde \cF_i^j$;
\item $\cO_i^j(x)\subset\cP_i^j(x)$, $x\in M_i$, and $\tilde \cO_i^j(x)\subset\tilde \cP_i^j(x)$, $x\in \tilde M_i$;
\item $D\pi_i(\tilde E_i^j)=E_i^j$;
\item $\pi_i^{-1}(\cP_i^j(\pi_i(x)))=\tilde \cP_i^j(x)$ and $\pi_i(\tilde \cF_i^j(x))\subset \cF_i^j(\pi_i(x))$, $x\in M_i$; 
\end{enumerate}

\begin{lemma}
\label{lemma_induction}
For all $j=0,1,\ldots k$ and all $T\in A^*\Gamma_i^j$ we have $T(\tilde \cP_i^{j+1}(x))=\tilde \cP_i^{j+1}(T(x))$ and $DT(\tilde E_i^j(x))=\tilde E_i^j(T(x))$, $x\in \tilde M_i$.
\end{lemma}
\begin{proof}
It is sufficient to show that if $\psi\in\tilde V_i^{j+1}$ then $\psi\circ T\in\tilde V_i^{j+1}$ for all $T\in A^*\Gamma_i^j$ which implies that $\tilde V_i^{j+1}=\tilde V_i^{j+1}\circ T$. Indeed, if we have that then
\begin{multline*}
\tilde \cP_i^{j+1}(T(x))=\bigcap_{\psi\in\tilde V_i^{j+1}}\psi^{-1}(\psi(T(x)))=T\left(\bigcap_{\psi\in\tilde V_i^{j+1}}(\psi\circ T)^{-1} \big((\psi\circ T)(x))\big)\right)\\
=T\left(\bigcap_{\psi\in\tilde V_i^{j+1}}\psi^{-1}(\psi(x))\right)=T(\tilde \cP_i^{j+1}(x))
\end{multline*}
Similarly, we would have $DT(\tilde E_i^j(x))=\tilde E_i^j(T(x))$ (cf. the proof of Lemma~\ref{new_lemma}).

Let $h_\#\colon A^*\Gamma_1^j\to A^*\Gamma_2^j$ be the isomorphism given by conjugation, $h_\#(T)=\tilde h\circ T\circ \tilde h^{-1}$. To complete the proof we have to show that 
if $(\psi_1,\psi_2)\in\tilde V_i^{j+1}$ then $(\psi_1\circ T,\psi_2\circ h_\#(T))\in\tilde V_i^{j+1}$. We have
$$
\psi_2\circ h_\#(T)\circ\tilde h=\psi_2\circ\tilde h\circ T=\psi_1\circ T
$$
and it remains to check that the function $\psi_1\circ T$, $T\in A^*\Gamma_i^j$, is $A^*\Gamma_i^{j+1}$-equivariant. Indeed, for $S\in A^*\Gamma_i^{j+1}$ we have
$$
\psi_i\circ T\circ S=\psi_i\circ (T\circ S\circ T^{-1})\circ T=\psi_i\circ T,
$$
where the last equality holds because $T\circ S\circ T^{-1}\in A^*\Gamma_i^{j+1}$ by Lemma~\ref{lemma_normal}.
\end{proof}

\begin{lemma}
\label{lemma_sep}
Let $X$ and $Y$ be finite subsets of $M_i$. If $X\cap \cP_i^{j+1}(Y)=\varnothing$ then there exists a function $\psi\in V_i^{j+1}$ such that $\psi|_X=0$ and $\psi|_Y=1$.
\end{lemma}
\begin{proof}
The proof is based on the observation that if $\psi\in V_i^{j+1}$ then $\phi\circ \psi \in V_i^{j+1}$ for any $C^r$ function $\phi\colon\R\to\R$.

First consider the case when $X=\{x\}$ and $Y=\{y_1, y_2,\ldots y_p\}$. Then because $x\notin \cP_i^{j+1}(Y)$ it can be separated from every point in $Y$ by a function from $V_i^{j+1}$; that is, for all 
$t=1,\ldots p$ there exists $\psi_t\in V_i^{j+1}$ such that $\psi_t(x)\neq\psi_t(y_t)$. By replacing $\psi_t$ with appropriate linear combination $A\psi_t+B$, we can assume that $\psi_t(x)=0$ and $\psi_t(y_t)=1$. Now let
$$
\psi_x=\sum_{t=1}^p\psi_t^2
$$
Then we have $\psi_x(x)=0$ and $\psi_x(y_t)\ge 1$ for all $t=1,\ldots p$. Finally we replace $\psi_x$ with $\phi\circ \psi_x$, where $\phi$ is a $C^r$ function such that $\phi(0)=0$ and $\phi(\xi)=1$ for all $\xi\ge 1$. This completes the proof in the case when $X=\{x\}$.

In the general case we have $X=\{x_1, x_2,\ldots x_q\}$ and $Y=\{y_1, y_2,\ldots y_p\}$ and we ca apply the above construction to each $x_s$, $s=1,\ldots q$, to obtain a function $\chi_s\in V_i^{j+1}$ such that $\chi_s(x_s)=0$ and $\chi_s(y_t)=1$ for all $t=1,\ldots p$. Consider
$$
\chi=\sum_{s=1}^q(1-\chi_s)^2
$$
Obviously, $\chi(y_t)=0$ for all $t$ and $\chi(x_s)\ge 1$ for all $s$. We can use the $C^r$ function $\phi$ to define the posited separating function as $\psi=1-\phi\circ\chi$.
\end{proof}

\begin{lemma}
\label{lemma_P_orbit}
$\tilde \cP_i^{j+1}(\tilde \cO_i^j(x))=\tilde \cP_i^j(x)$, $x\in\tilde M_i$ and $ \cP_i^{j+1}( \cO_i^j(x))= \cP_i^j(x)$, $x\in M_i$, for all $j=0,1,\ldots k$.
\end{lemma}
\begin{proof}
The inclusion $\tilde \cP_i^{j+1}(\tilde \cO_i^j(x))\subset \tilde \cP_i^j(x)$ is straigtfoward. Indeed $\tilde \cO_i^j(x)\subset \tilde \cP_i^j(x)$ and $\tilde \cP_i^{j+1}(y)\subset \tilde \cP_i^j(y)=\tilde \cP_i^j(x)$ for all $y\in \tilde \cO_i^j(x)$.

Assume the reverse inclusion does not hold. Then there exists a point $x$ and $y\in\tilde \cP_i^j(x)$ such that $y\notin \tilde \cP_i^{j+1}(\tilde \cO_i^j(x))$. By Lemma~\ref{lemma_induction} the set $\tilde \cP_i^{j+1}(\tilde \cO_i^j(x))$ is $A^*\Gamma_i^j$-invariant and hence $\tilde \cO_i^j(y)\cap \tilde \cP_i^{j+1}(\tilde \cO_i^j(x))=\varnothing$. Then we also have $\cO_i^j (\pi_i(y))\cap \cP_i^{j+1}( \cO_i^j(\pi_i(x)))=\varnothing$ and we can apply Lemma~\ref{lemma_sep} to $\cO_i^j (\pi_i(y))$ and $\cO_i^j (\pi_i(x))$ and obtain a function $\psi\in V_i^{j+1}$ such that $\psi|_{\cO_i^j (\pi_i(x))}=0$ and $\psi|_{\cO_i^j (\pi_i(y))}=1$. We now lift $\psi$ to $\tilde M_i$ and consider the finite sum
$$
\bar\psi=\sum_{[T]\in A^*\Gamma^j_i/ A^*\Gamma^{j+1}_i} \psi\circ\pi_i\circ T
$$
Note that the summands are well-defined because $\psi\circ \pi_i\in\tilde V_i^{j+1}$ and, hence, are $A^*\Gamma^{j+1}_i$-equivariant. Notice that $\bar\psi|_{\tilde \cO_i^j (x)}=0$ and $\bar\psi|_{\tilde \cO_i^j (y)}=|A^*\Gamma^j_i/ A^*\Gamma^{j+1}_i|>0.$ Finally notice that for any $[S]\in A^*\Gamma^j_i/ A^*\Gamma^{j+1}_i$
$$
\bar\psi\circ S=\sum_{[T]\in A^*\Gamma^j_i/ A^*\Gamma^{j+1}_i} \psi\circ\pi_i\circ T\circ S=\sum_{[T]\in A^*\Gamma^j_i/ A^*\Gamma^{j+1}_i} \psi\circ\pi_i\circ T=\bar\psi
$$
Hence $\bar\psi$ belongs to $\tilde V_i^j$ and separates $x$ and $y$ which yields a contradiction.
\end{proof}
We finally arrive at the main lemma of this subsection. Recall that $\tilde \cF_i=\tilde \cF^{k+1}_i$.
\begin{lemma}
\label{lemma_invariance}
For all $j=0,1,\ldots k$, $\tilde \cF_i^j=\tilde\cF_i$ and $\tilde \cF_i$ is a $A^{-1}\Gamma_i$-invariant partition of $M_i$, $i=1,2$.
\end{lemma}

\begin{proof}
Recall that $cc_x$ stands for ``connected component of $x$." Applying the previous lemma we have
\begin{eqnarray*}
	\tilde\cF_i^j(x)&=&cc_x\left(\tilde \cP_i^j(x)\right)=cc_x\left(\bigcup_{y\in\tilde\cO_{i}^j(x)}\tilde\cP_i^{j+1}(y)\right)\\
	&=&cc_x\left(\bigcup_{[T]\in A^*\Gamma_i^j/A^*\Gamma_i^{j+1}}T(\tilde\cP_i^{j+1}(x))\right)
	=cc_x\left(\tilde \cP_i^{j+1}(x)\right)=\tilde\cF_i^{j+1}(x)
\end{eqnarray*}
where the first equality in the second line is due to invariance of $\tilde \cP_i^{j+1}(y)$ under the action of $A^*\Gamma_i^{j+1}$.
By induction on $j$ we conclude that $\tilde\cF_i^{j}(x)=\tilde\cF_i^{k+1}(x)$. In particular $\tilde \cF_i^0=\tilde \cF_i$. It remains to recall that $\tilde \cF_i^0$ is $A^*\Gamma_i^0=A^{-1}\Gamma_i$-invariant.
\end{proof}

Using higher iterates of expanding maps we can prove, using exactly the same arguments, that partitions $\tilde \cF_i$ are invariant under the action $A^{-n}\Gamma_i$ for all $n\ge 1$. Hence we have the following corollary. (Recall that $A^{-\infty}\Gamma_i=\cup_{n\ge 1}A^{-n}\Gamma_i$.)
\begin{corollary}
The partition $\tilde \cF_i$ of $\tilde M_i$ is invariant under the action of $A^{-\infty}\Gamma_i$.
\label{cor_invariance}
\end{corollary}

\subsection{Upgrading to a foliation} Now we will prove that $\tilde \cF_i$ is, in fact, a $C^r$ smooth foliation.

\label{sec_up_foliation}

Consider the dimension function $\dim \tilde E_i$ and let $m_i=\min_{x\in\tilde M_i} \dim \tilde E_i(x)$. Pick a point $x\in\tilde M_i$ such that $\dim\tilde E_i(x)=m_i$. Then, by definition of $\tilde E_i(x)$ we can find functions $\psi_1,\psi_2,\ldots \psi_{d-m_i}\in \tilde V_i$ such that
$$
\tilde E_i(x)=\bigcap_{j=1}^{d-m_i} \ker d_x\psi_j
$$
From continuity of $d_x\psi_j$ and the fact that $m_i$ is minimal dimension, we also have the same formula
$$
\tilde E_i(y)=\bigcap_{\psi\in\tilde V_i}\ker d_y\psi=\bigcap_{j=1}^{d-m_i} \ker d_y\psi_j
$$
for all $y$ in a sufficiently small open neighborhood $B$ of $x$. 

Consider the map $\Psi\colon B\to \R^{d-m_i}$ given by $\Psi(y)=(\psi_1(y),\psi_2(y),\ldots \psi_{d-m_i}(y))$. It is clear that the plaque $\Psi^{-1}(\Psi(y))$ is tangent to $\tilde E_i$ at every point of the plaque. By choosing $B$ appropriately we may assume that the plaques $\Psi^{-1}(\Psi(y))$  are path-connected for all $y\in B$.
\begin{lemma} For every $y\in B$ we have $\Psi^{-1}(\Psi(y))=\tilde \cF_i(y)\cap B$.
\label{lemma_up_foliation}
\end{lemma}
\begin{proof} If a point $z\in B$ does not belong to a plaque $\Psi^{-1}(\Psi(y))$ then one of the functions $\psi_j$ separates $z$ and $y$. Hence $z\notin\tilde \cP_i(y)\supset\tilde \cF_i(y)$.

Now take $z\in \Psi^{-1}(\Psi(y))$ and consider any function $\psi\in\tilde V_i$. Connect $z$ to $y$ by a path. If $\psi(z)\neq\psi(y)$ then for some point $q$ on the path the restriction of $\psi$ to this path have non-zero derivative and, hence, $\tilde E_i(q)\nsubset \ker d_q\psi$ giving a contradiction. Hence $\psi(z)=\psi(y)$ for all $\psi\in\tilde V_i$, which implies that $\Psi^{-1}(\Psi(y))\subset\tilde P_i(y)$. Thus we also have $\Psi^{-1}(\Psi(y))\subset \cF_i(y)$ because $\Psi^{-1}(\Psi(y))$ is connected.
\end{proof}

Now we have that the restriction $\tilde \cF_i|_B$ is a foliation and we would like to spread the foliation structure to the whole $\tilde M_i$. For that we have to see that $A^{-n}\Gamma_i(B)=\tilde M_i$. If we have it, then
 using invariance under $A^{-n}\Gamma_i$ provided by Corollary~\ref{cor_invariance}, we can conclude that $\cF_i$ has $C^r$ foliation structure in the neighborhood of every point. And, hence, $\cF_i$ is indeed a $C^r$ foliation.

Recall that the action of of $A^{-n}\Gamma_i$ on $\tilde M_i$ is conjugate via $\tilde h_i$ to the action by left translations by $A^{-n}(\Gamma)\subset N$ on $N$. To guarantee that $A^{-n}\Gamma_i(B)=\tilde M_i$ it suffices to choose a sufficiently large $n$ so that the set $\tilde h_i(B)$ covers a fundamental domain of the lattice $A^{-n}(\Gamma)$.

\subsection{Upgrading to a fibration and completing the proof}
Now we have that both $\tilde \cF_1$ and $\tilde \cF_2$ are $C^r$ foliations. We also have that $\tilde h(\tilde \cF_1)=\tilde \cF_2$ and, hence, by Invariance of Domain, these foliations have the same dimension. It remains to show that $\pi_i(\tilde \cF_i)$ (which are clearly also $C^r$ foliations) are, in fact, $C^r$ fibrations. We also need to show that the fibers $\pi_i(\tilde \cF_i)(x)$ and the base of the fibrations are homeomorphic to nilmanifolds and that the induced conjugacy on the base is a $C^r$ diffeomorphism. To do that we go to linearized dynamics on $N/\Gamma$ similarly to the argument in Section~\ref{sec_fib}.

Let $\tilde F=\tilde h_1(\tilde \cF_1)=\tilde h_2(\tilde \cF_2)$. Then $\tilde F$ is a topological foliation with closed leaves which is invariant under the expanding automorphism $A\colon N\to N$. By Corollary~\ref{cor_invariance} foliation $\tilde F$ is also invariant under left translations by $A^{-\infty}(\Gamma)$. Because $\tilde F$ is continuous and $A^{-\infty}(\Gamma)$ is dense we conclude that $\tilde F$ is invariant by all left translations on $N$. This allows us to argue that $\tilde F(id_N)$ is a group. Indeed for all $x, y\in \tilde F(id_N)$ we have $\tilde F(xy)=x\tilde F(y)=x\tilde F(id_N)=\tilde F(x id_N)=\tilde F(id_N)$, and similarly for all $x\in \tilde F(id_N)$ we have $\tilde F(x^{-1})=x^{-1}\tilde F(id_N)=x^{-1}\tilde F(x)=\tilde F(x^{-1}x)=\tilde F(id_N)$. We can now apply apply Cartan's closed subgroup (see \eg~\cite{hall}) theorem to conclude that $\tilde F(id_N)$ is a Lie subgroup of $N$. 

So we denote the leaf through identity by $G=\tilde F(id_N)$.
Hence, using translation invariance again, we conclude that $\tilde F$ is a smooth foliation by cosets of $G$. 

\begin{lemma} Let $F$ be the projection of $\tilde F$ on $N/\Gamma$. Then each leaf of $F$ is either compact or it ``accumulates on itself", that is, there exists $x\in N/\Gamma$ such that for arbitrarily small neighborhood $B$ of $x$ the intersection $F(x)\cap B$ has infinitely many connected components.
\label{lemma_ratner}
\end{lemma}
\begin{proof}
The leaves of $F$ are orbits of the action of $G$, which is a nilpotent Lie group on $N/\Gamma$.
So one can refer to Ratner theory, specifically to~\cite{ratner}, which gives, in particular, that the closures of  orbits of such a unipotent action are affinely embedded nilmanifolds. Hence each orbit is either compact or dense in it's higher dimensional closure, which implies the needed recurrence. It also not so hard to derive this lemma from earlier work of Parry on homogeneous flows on nilmanifolds: one needs to choose orbits which escape to infinity in non-compact leaves and use~\cite[Theorem 5]{parry}. 

However the lemma can also be derived from more basic topological dynamics using work of Ellis and Furstenberg on distal actions~\cite{ellis, Fur}, which we proceed to explain. It well-known and simple fact that the nilpotent action of $G$ on $N/\Gamma$ is distal. (It follows from the fact that a nil-translation is an iterated isometric extension). Based on work of Ellis, Furstenberg proved that a distal actions can be decomposed into a disjoint union of minimal sets~\cite[Theorem 3.2]{Fur}. Hence each leaf of $F$ is either compact or has a non-trivial closure and is dense in the closure and hence recurrent.
\end{proof}

\begin{lemma} There exists a non-empty open set $U\subset N/\Gamma$ such that each leaf of $F$ that meets $U$ is compact.
\label{lemma_open_set}
\end{lemma}

\begin{proof}
We begin by noticing that the properties of being compact and to ``to accumulate on itself" are topological. Hence we have the property given by the Lemma~\ref{lemma_ratner} on the non-linear side as well by applying $h_i^{-1}$: for all $x\in M_i$ either $\pi_i(\tilde F_i)(x)$ is compact or for all small neighborhoods $B$ of $x$ the intersection $(\pi_i(\tilde F_i)(x))\cap B$ has infinitely many connected components.

Now we will argue in the same way as in Subsection~\ref{sec_up_foliation}, but on $M_i$ instead of $\tilde M_i$. 
Let $U_i\subset M_i$ be the set where $\dim E_i$ achieves its minimum $m_i$. Recall that the dimension function $x\mapsto \dim E_i(x)$ is upper semicontinuous, which implies that $U_i$ is open. 

Take a point $x\in U_i$.  Then we can construct a foliation chart for $\cF_i$ which we denote by $\Psi\colon B\to \R^{d-m_i}$ about $x$, $B\subset U_i$, such that $\Psi^{-1}(\Psi(x))$ is a connected subset of $\cF_i(x)$ and for all $z\in B$ which do not belong to $\Psi^{-1}(\Psi(x))$ we have $z\notin \cP_i(x)$ (see the proof of Lemma~\ref{lemma_up_foliation}). Then we have
$$
(\pi_i(\tilde \cF_i)(x))\cap B\subset \cF_i(x)\cap B\subset \cP_i(x)\cap B=\Psi^{-1}(\Psi(x))
$$
On the other hand, recalling that $D\pi_i(\tilde E_i)=E_i$ and the discussion at the  beginning of Subsection~\ref{sec_up_foliation}, we have that $\pi_i(\tilde \cF_i)$ is a $(d-m_i)$-dimensional foliation; that is, it has the same dimension as the plaque $\Psi^{-1}(\Psi(x))$, hence $\pi_i(\tilde \cF_i)(x)\cap B=\Psi^{-1}(\Psi(x))$. Therefore $\pi_i(\tilde \cF_i)(x)\cap B$ has only one connected component and we conclude, by dichotomy of Lemma~\ref{lemma_ratner}, that $\pi_i(\tilde \cF_i)(x)$ is compact. Going back to the foliation $F$ via $h_i$ we obtain the same conclusion: all leaves of $F$ which meet the non-empty open set $U=h_i(U_i)$ are compact.
\end{proof}

\begin{lemma} Group $G$ is a normal subgroup of $N$.
\end{lemma}

\begin{proof}
By Lemma~\ref{lemma_open_set} there exists a small open ball $B\subset U \subset N/\Gamma$ such that every leaf $F(x\Gamma)=Gx\Gamma$, $x\in B$, is compact. We consider the stabilizer group $\Gamma_x$ of the leaf $\tilde F(x)$.
$$
\Gamma_x=\{\gamma\in\Gamma :\,\, Gx=Gx\gamma\}=\{\gamma\in\Gamma :\,\, x\gamma x^{-1}\in G\}=\Gamma\cap x^{-1}Gx\subset\Gamma
$$
Thus $F(x\Gamma)=\tilde F(x)/\Gamma_x$ is homeomorphic to $x^{-1}Gx/\Gamma_x$ and, hence $\Gamma_x$ is a cocompact lattice in $x^{-1}Gx$, $x\in B$.

Now assume that for some $x_0, x_1\in B$ we have $x_0^{-1}Gx_0\neq x_1^{-1}Gx_1$. Then we can find a path $x_t\in B$, $t\in I$, such that $x_t^{-1}Gx_t$, $t\in I$, are all mutually distinct subgroups of $N$. (Indeed, just connect $x_0$ to $x_1$ by a path in $B$ and then choose a small subpath in a neighborhood of a point where $x_t^{-1}Gx_t$ varies infinitesimally linearly with $t$.) 

We can also see that all $\Gamma_{x_t}=\Gamma\cap x_t^{-1}Gx_t$ are mutually distinct. To see that notice that, because the exponential map provides a one-to-one correspondence between subalgebras of the Lie algebra of the simply connected nilpotent Lie group $N$ and its connected Lie subgroups, we intersection subgroup $(x_t^{-1}Gx_t)\cap (x_s^{-1}Gx_s)$, $s\neq t$ is at least codimension one in both $x_t^{-1}Gx_t$ and $x_s^{-1}Gx_s$. Because $\Gamma_{x_t}$ is cocompact in $x_t^{-1}Gx_t$ it must have a non-trivial image in the non-compact quotient space $x_t^{-1}Gx_t/((x_t^{-1}Gx_t)\cap (x_s^{-1}Gx_s))$. Hence, indeed $\Gamma_{x_t}$ contains elements which are not in $\Gamma_{x_s}$. 

Thus we have obtained an uncountable family $\Gamma_{x_t}$, $t\in I$, of mutually distinct subgroups of $\Gamma$, which gives a contradiction. Indeed, $\Gamma$ is finitely generated and nilpotent and, hence, any subgroup of $\Gamma$ is also finitely generated. So $\Gamma$ only have countably many distinct subgroups. We conclude that $x_0^{-1}Gx_0= x_1^{-1}Gx_1$ for all $x_0, x_1\in B$. Hence $(x_1x_0^{-1})^{-1}Gx_1x_0^{-1}=G$ for all $x_1x_0^{-1}$ is a small neighborhood of $id_N$. And because such neighborhood generates $N$ we can conclude that $G$ is normal.
\end{proof}

Let $\Gamma_G=\Gamma\cap G$. Normality of $G$ implies that $\Gamma_x=\Gamma_G$ for all $x$ and we have that $\Gamma_G$ is a cocompact lattice in $G$ by Lemma~\ref{lemma_open_set}. Again using normality of $G$ (and $\Gamma_G$ in $\Gamma$) it is easy to check that the quotient homomorphism $N\to G\backslash N$ induces a well-defined fibration map $p\colon N/\Gamma\to\bar M$ with compact nilmanifold base $\bar M=(G\backslash N)/(\Gamma_G\backslash \Gamma)$ and nilmanifold fiber $G/\Gamma_G$.
Conjugating back to $M_i$ we obtain the posited fibrations $p_i\colon M_i\to\bar M_i$ whose fibers are the leaves of $\pi_i(\tilde \cF_i)$ which are homeomorphic to $G/\Gamma_G$ and whose base $\bar M_i$ is homeomorphic to $\bar M$. Note that $p_i$ are $C^r$ smooth because we already have that $\pi_i(\tilde \cF_i)$ are $C^r$.

It remains to check that $\bar h \colon \bar M_1 \to \bar M_2$ induced by $h$ is a $C^r$ diffeomorphism. This remaining argument follows closely the corresponding argument in Section~\ref{sec_fib}. The only difficulty comes from the fact that we still do not know that $T(\pi_i(\tilde \cF_i))=E_i$ (and that $\cF_i$ is a foliation). However, we do know, from the proof of Lemma~\ref{lemma_open_set}, that $E_i=T\cF_i=T(\pi_i(\tilde \cF_i))$ on an open and dense set --- the set where $E_i$ achieves the minimal dimension $m_i$. Indeed, recall that $m_i$ denotes the minimal dimension of $E_i$. We have shown that $m_1=\dim\cF_1=\dim\cF_2=m_2$. Let $U_i=\{x\in M_i: \dim E_i(x)=m_i\}$. Recall that $U_i$ is open (cf. the discussion of Subsection~4.4). Further if $x\in U_i$ and $f_i(y)=x$ then, using $ V_i\circ f_i\subset V_i$, we have
\begin{multline*}
Df_i(E_i(y))=Df_i\left(\bigcap_{\psi\in V_i}\ker d_y\psi \right)\subset Df_i\left(\bigcap_{\psi\in V_i}\ker d_y(\psi\circ f_i) \right)\\
=\bigcap_{\psi\in V_i}Df_i(\ker d_y(\psi\circ f_i))=\bigcap_{\psi\in V_i}\ker d_x\psi=E_i(x)
\end{multline*}
Hence, because $\dim E_i(x)$ is minimal the above inclusion is, in fact, equality and $y\in U_i$. We obtain $f_i^{-1}(U_i)\subset U_i$, which implies that $U_i$ is dense in $M_i$. Therefore we can pick a point $x\in U_1$ such that $h(x)\in U_2$. Consequently both of these points admit nice foliation charts and we can show, repeating verbatim the arguments of the last paragraph of Subsection 4.1, that $\bar h$ is a $C^r$ diffeomorphism on a neighborhood $B$ of $p_1(x)$.

Now recall that $\bar h$ conjugates the induced $C^r$ expanding maps, $\bar h\circ\bar f_1=\bar f_2\circ \bar h$. We can lift all the maps to the universal covers and express the lift of $\bar h$ as follows:\
$$
\tilde{\bar h}=\tilde{\bar f}_1^n\circ \tilde{\bar h}\circ \tilde{\bar f}_2^{-n}
$$
If $\tilde B$ is the lift of $B$, then the above equation implies that $\tilde{\bar h}$ is a $C^r$ diffeomorphism on $\tilde{\bar f}_2^{n}(\tilde B)$. For a sufficiently large $n$ the set $\tilde{\bar f}_2^{n}(\tilde B)$ contains a fundamental domain of the cover and we conclude that $\bar h$ is indeed a $C^r$ diffeomorphism.

\begin{remark} Once the proof is finished we can actually conclude that $\cF_i=\pi_1(\tilde\cF_i)$ and $E_i=T\cF_i$. This fact would be very helpful to have in the course of the proof, however it was out reach and we only can obtain it a posteriori. To see that $\cF_i=\pi_1(\tilde\cF_i)$ and $E_i=T\cF_i$ one can characterize the space of functions $V_i$ as the space of $C^r$ functions which are constant on the fibers of $p_i$. Such characterization easily follows from the fact that $\bar h$ is a $C^r$ diffeomorphism. Note, however, that this fact is not needed in the statement of Theorem~\ref{rigexpthm}.
\end{remark}

\section{Proofs of Corollaries} 
\begin{proof}[Proof of Corollary~\ref{cor1}]
We denote by $L$ the linear endomorphism to which both $f_1$ and $f_2$ are conjugated.

By passing to the second iterate we may assume $\textup{Jac}(f_i)>0$, $i=1,2$. Let $\phi_i=-\log \textup{Jac}(f_i)$. By Theorem \ref{rigexpthm} we have $C^r$ fibrations (with connected fiber) $p_i:M_i\to\bar M_i$ and functions $\bar \phi_i:\bar M_i\to\R$ such that  $\bar \phi_i\circ p_i$ is cohomologous to $\phi_i$ and the induced conjugacy $\bar h:\bar M_1\to \bar M_2$, $\bar h\circ p_1=p_2\circ h$, 
 is a $C^r$ diffeomorphism.
 
If $\dim \bar M_1=0$, then $\bar\phi_1$ is constant and, hence, $\phi_1$ is cohomologous to a constant. Then the equilibrium state for $\phi_1$, which is the absolutely continuous measure equals the equilibrium state for the constant function which is the entropy maximizing measure \cite{B}, contradicting the assumption of the corollary.

If $\dim \bar M_1=d$, then $p_1$ and $p_2$ are diffeomorphisms (in fact, identity diffeomorphisms) and, hence, $h$ is a $C^r$ diffeomorphism since $h=p_2^{-1}\circ\bar h\circ p_1$.

It remains to consider the case when $0<\dim\bar M_1<d$. However, this is impossible due to irreducibility. Indeed from the proof of Theorem~\ref{rigexpthm} in Section~4.1 it is clear that $L$ leaves invariant a torus of a positive dimension $m<d$, which contradicts to irreducibility. We can also provide an alternative self-contained short argument below.

Abbreviate $M=M_1$ and $\bar M=\bar M_1$. Let $x$ be a fixed point of $f_1$ and let $F$ be the fiber of $p_1$ which contains $x$. Recall that, by Theorem~\ref{rigexpthm}, $\bar M$ supports an expanding map $\bar f_1$  and, hence, is aspherical. Therefore the fundamental groups fit into the short exact sequence
$$
0\to \pi_1(F)\to\pi_1(M)\to\pi_1(\bar M)\to 0
$$
Note that taking tensor product with $\R$ leaves the sequence exact.

Because $f_1(F)=F$ we have $(f_1)_*(\pi_1(F))=L_*(\pi_1(F))<\pi_1(F)<\pi_1(M)\simeq\Z^d$. Since $\dim\bar M<d$ we have that $\dim F>0$ and $F$ is compact and also aspherical (because it supports the expanding map $f_1|_F$). It follows that $\pi_1(F)\otimes \R$ gives a non-zero rational invariant subspace for $L$. Because $L$ is irreducible we conclude that  $\pi_1(F)\otimes \R=\R^d$. Hence $\pi_1(\bar M)\otimes\R=0$, \ie $\pi_1(\bar M)$ is torsion finitely generated abelian group, hence, finite. But any closed aspherical manifold of dimension >0 has an infinite fundamental group, a contradiction.
\end{proof}

\begin{proof}[Proof of Corollary~\ref{rigexpcor}]
By classification of expanding maps, manifolds  $M_i$ are homeomorphic to infranilmanifolds. Therefore we can pass to the nilmanifold covers and, accordingly, pass to the lifts of expanding maps. It is easy to see that the very non-algebraic assumption still holds for the lifted maps. From now on we assume that $M_i$ are homeomorphic to nilmanifolds and, hence, Theorem~\ref{rigexpthm} applies.
	
	Recall that $f_i$ are $C^{r+1}$, $r\geq 1$,  very non-algebraic expanding maps and $h:M_1\to M_2$ is a conjugacy. We apply Theorem~\ref{rigexpthm} to $f_i$ and $r$ (not $r+1$!). Let $p_i:M_i\to\bar M_i$, $\bar f_i:\bar M_i\to\bar M_i$, $\bar h:\bar M_1\to\bar M_2$ be the $C^r$ maps given by Theorem~\ref{rigexpthm}. We shall show that $\dim(\bar M_i)=\dim(M_i)$. Then we would have $h=\bar h$ and, by  Theorem~\ref{rigexpthm}, the conjugacy is $C^r$.
	

	  Assume that $\dim(M_i)-\dim(\bar M_i)=m>0$. Recall that by Theorem~\ref{rigexpthm} the fibers $F_{i,x}=p_i^{-1}(p_i(x))$ are nilmanifolds and, hence, are orientable. Moreover the fibers can be simultaneously coherently oriented because the base space $\bar M_i$ is also an orientable nilmanifold. We fix a choice of orientation on fibers and on the base. The expanding map $f_i$ does not necessarily preserve any of the orientations. (And we cannot pass to finite iterates because such operation would not preserve the ``very non-algebraic condition.'') Let $\boldsymbol d$ be the absolute value of the degree of the map between the fibers 
	$$
	\boldsymbol d=\left|\deg(f_i|_{F_{i,x}}\colon F_{i,x}\to F_{i, f_i(x)})\right|
	$$
Note that $\boldsymbol d$ is indeed independent of $x$ by continuity and is independent of $i$ because $f_i$ are conjugate. Further, if $\dim F_{i,x}>0$ then $\boldsymbol d>1$ because the  expanding map on the fiber through a fixed point is a self cover of degree $>1$.
	

	
	In the rest of the proof write $J$ to {\it denote the absolute value of the Jacobian} of a map --- $Jf:=|\textup{Jac}(f)|$. Let $\psi_i=\log (Jf_i|_{\ker Dp_i})$. We note that these functions are only $C^{r-1}$ because the distributions $\ker Dp_i$ are merely $C^{r-1}$.

First we pick Riemannian metrics on $\bar M_i$, $i=1,2$, so that $\bar h$ is volume preserving (\eg an isometry) and, hence,
 $$
 \log J\bar f_1=\log J\bar f_2 \circ \bar h
 $$
 Then pick a smooth connections $\mathcal E_i$ for $p_i$ (subbundles transverse to $\ker Dp_i$) and then lift the Riemannian metrics from $\bar M_i$ to $\mathcal E_i$. Then consider Riemannian metrics on $M_i$ which are direct sums of metrics on $\ker Dp_i$ and the lifted metrics on $\mathcal E_i$. By construction, the differential $Df_i$ have upper-triangular form and we have
	$$
 \psi_i = 	\log J f_i - \log J\bar f_i\circ p_i
	$$
The Livshits Theorem for expanding maps together with the assumption on Jacobians at periodic points imply that $\log J f_1$ is cohomologous to $\log J f_2\circ h$. Note that $\log J f_i$ are $C^r$ functions. Hence, by the main property of $(p_i,\bar f_i,\bar h)$ given by Theorem~\ref{rigexpthm}, we have that $\log J f_i$ is $f_i$-cohomologous to a $C^r$ function which is constant on the fibers. Because $\log J\bar f_i\circ p_i$ are also constant on the fibers, it follows that $\psi_i$ are cohomologous to $\bar\psi_i\circ p_i$.


	In other words, there exist $C^r$ function $u_i$ such that 
	$$
	\log (Jf_i|_{\ker Dp_i})-u_i+u_i\circ f_i=\bar\psi_i\circ p_i
	$$ 
	Therefore by replacing  the volume form $\omega$ on the fibers $F_{1,x}$ with the volume form 
	$$
	\bar\omega=e^{u_1}\omega$$ 
	we can assume that the absolute value of the Jacobian of $f_1|_{\ker Dp_1}$ equals to $e^{\bar\psi_1\circ p_1}$.  Denote by $vol(F_{1,x})$ the total $\bar\omega$-volume of $F_{1,x}$. For any $x\in M_1$ we have
	 \begin{multline*}
	 e^{\bar\psi_1\circ p_1(x)}=\frac{1}{vol(F_{1,x})}\int_{F_{1,x}} e^{\bar\psi_1\circ p_1}\bar\omega 
	 =\frac{1}{vol(F_{1,x})}\int_{F_{1,x}} Jf_1|_{\ker Dp_1}\bar\omega=\boldsymbol d
	 \end{multline*}
Hence for every periodic point $x$, $f_1^kx=x$ we have that $Jf_1^k|_{\ker Dp_1}=\boldsymbol d^k$. This means that either $\boldsymbol d^k$ or $(-\boldsymbol d)^k$ belongs to the spectrum of 
$$
\bigwedge^m Df_1^k(x)
$$
 which contradicts to $f_1$ being very non-algebraic. We conclude that $m=0$, \ie $\dim(\bar M_1)=\dim(M_1)$ and we are done.
\end{proof}

Note that even though the regularity of $f_i$ is $r+1$, we use Theorem \ref{rigexpthm} with regularity $r$ because we work with Jacobian of $f_i$.

\section{Examples}

\begin{example}[Basic example]
\label{ex1}
Here we give an explicit example where non-trivial fibrations $p_i\colon M_i\to\bar M_i$, $i=1,2$, with $\dim \bar M_i\neq 0, \dim M_i$ appear. Consider the expanding maps $L, f\colon \T^2\to\T^2$ given by $L(x,y)=(2x,2y)$ and $f(x,y)=(g(x), 2y)$, where $g$ is conjugate to $\times 2$ map via nowhere differentiable conjugacy $h_0$, $h_0\circ g= 2h_0$. For simplicity we may assume that $g(0)=0$ and $g'(0)<2$. Then $h=(h_0, id_{S^1})$ is the conjugacy between $f$ and $L$. Recall that fibrations $p_i$ arise from the space of pair of $C^r$ functions $(\psi_1,\psi_2)$ which satisfy $\phi_1=\psi_2\circ h$, \ie 
$$
\psi_1(x,y)=\psi_2(h_0(x),y)
$$
Clearly any $C^r$ function $\psi_1(x,y)=\psi(y)$ belong to this space. We will show that these are the only functions which could appear. Then, it immediately follows that $p_1(x,y)=p_2(x,y)=y$. That is, $p_i$ are circle fibrations over $S^1$.

Denote by ${\partial}_{inf} h_0$ the lower derivative of $h_0$ defined via $\liminf$. All periodic points which spend sufficiently large proportion of time near $0$ have Lyapunov exponent $<\log2$. Such periodic points $p$ are dense in $S^1$ and it is easy to see that ${\partial}_{inf} h_0(p)=0$ for any such $p$. Hence differentiating the relation between $\psi_1$ and $\psi_2$ with respect to $x$ yields
$$
\frac{\partial}{\partial x}\psi_1(p,y)=\frac{\partial}{\partial x}\psi_2(h_0(p),y) \partial_{inf} h_0(p)=0
$$
for a dense set of $p$. Hence, indeed, $\psi_1$ and $\psi_2$ are functions of $y$ only.

\end{example}

\begin{remark}\label{remark62}
Any primitive vector $(m,n)\in\Z^2$ yields a fibration $S^1\to\T^2\to S^1$ whose fibers in the universal cover $\R^2$ are lines parallel to the vector $(m,n)$. This gives infinitely different fibrations each of which is preseved by the conformal map $L$ from Example~\ref{ex1}. Further, similarly to the construction of Example~\ref{ex1}, one can construct perturbations $f_{(m,n)}$ such that the fibration given by the Theorem~\ref{rigexpthm} is precisely the fibration coming from $(m,n)$. 
\end{remark}

\begin{example}[de la Llave example] Non-trivial fibration may appear in a more subtle way when Jacobians full periodic data match. Of course, this can only happen for expanding maps which are not very non-algebraic. The example presented here is due to de la Llave~\cite{dlL}.

Consider the maps
$$
L(x,y)=(dx, ay),\,\, d\ge 2, a\ge 2,
$$
and 
$$
f(x,y)=(dx+\alpha(y), ay)
$$
Then the conjugacy between $L$ and $f$ has the form
$$
h(x,y)=(x+\beta(y), y)
$$
where $\beta$ can be expressed explicitly as the series~\cite{dlL} 
$$
\beta(y)=\frac{1}{d}\sum_{i\geq 0}\frac{1}{d^i}\alpha(a^iy)
$$ 
Notice that $\beta$ is a Weierstrass function.
Let 
$$
r_0=\frac{\log d}{\log a}
$$ 
and let $r_0=n+\theta$ where $n_0\in\N_{0}$ and $\theta\in(0,1]$. 

 To analyze the regularity of $\beta$ there are several cases to consider which give different answers. 

\begin{lemma}
	Assume that $\alpha\in C^r$, $r=k+\delta$, $k\in\N_0$, $\delta\in[0,1)$ and let $r_0=n+\theta$, as above $n\in\N_0$ and $\theta\in(0,1]$, then
	\begin{enumerate}\label{regularitybeta}
		\item Case I: $r<r_0$ then $\beta\in C^r$;
		\item Case II: $r>r_0$, $r_0\notin\N$ then $\beta\in C^{r_0}$;
		\item Case III: $r>r_0$, $r_0=n+1\in\N$ then $\beta\in C^{n+x|\log x|}$; 
		\item Case IV: $r=r_0$ then $\beta\in C^{n+x^\theta|\log x|}$;
	\end{enumerate} 
In all cases, there is a generic set of $\alpha\in C^r$ where the regularity is optimal, in particular for such $\alpha$, $\beta\notin C^{r_0+\epsilon}$ for any $\epsilon>0$.
\end{lemma}
\begin{proof}
	We give the proof for Case IV, all other cases being analogous.
	
By term-wise differentiation we have that 
$$
\beta^{(n)}(y)=\frac{1}{d}\sum_{i\geq 0}\left(\frac{a^{n}}{d}\right)^i\alpha^{(n)}(a^iy)
$$ 
	 which is convergent because $r_0> n$.  Comparing the series for $\beta$ and $\beta^{(n)}$, clearly we can assume that $n=0$ because the argument for $n>0$ would be the same with $\beta^{(n)}$ in place of $\beta$.
	
	
	Let $A=\max|\alpha|$ and $C$ the $\theta$-H\"older constant for $\alpha$. Take $x\neq y$ and let $N$ be such that 
	$$\frac{1}{a^{N-1}}\leq|x-y|\leq\frac{1}{a^N}$$
	
	Then 
	$$|\beta(x)-\beta(y)|\leq\sum_{k=0}^{N-1}\frac{1}{a^{\theta k}}\left|\alpha(a^kx)-\alpha(a^ky)\right|+\sum_{k\geq N}\frac{1}{a^{\theta k}}\left|\alpha(a^kx)-\alpha(a^ky)\right|
	$$ 
	The first summand is smaller than $$C\sum_{k=0}^{N-1}\frac{1}{a^{\theta k}}a^{\theta k}|x-y|^\theta\leq CN|x-y|^\theta\leq C|x-y|^{\theta}|\left|\log|x-y|\right|$$
	 The second summand is smaller than 
	 $$
	 \frac{2A}{1-a^\theta}\frac{1}{a^{N\theta}}\leq C|x-y|^\theta
	 $$
	Hence we obtain the posited $x^\theta|\log x|$ modulus of continuity for $\beta$.
	
	On the other hand, if we assume, to simplify notation, that $\alpha(0)=0$, and say $\alpha(x)>0$ for $x>0$, and that $\liminf_{x\to 0}\frac{|\alpha(x)|}{|x|^\theta}>0$. Pick $\epsilon_0>0$ sufficiently small so that  $K=\inf_{|x|<\epsilon_0}({|\alpha(x)|}/{|x|^\theta})$ is positive. Then, taking $x>0$ very close to $0$ and $N>0$ first such that $a^Nx\geq \epsilon_0$, we obtain
	\begin{eqnarray*}
		|\beta(x)-\beta(0)|&\geq& \sum_{k=0}^{N-1}\frac{1}{a^{\theta k}}\alpha(a^kx)-\sum_{k\geq N}\frac{1}{a^{\theta k}}\alpha(a^kx)\\
		&\geq& \left(KN-\frac{2A}{1-a^\theta}\right)|x|^\theta \geq\left(KC_{\epsilon_0}|\log|x||-\frac{2A}{(1-a^\theta)\epsilon_0^\theta}\right)|x|^\theta
	\end{eqnarray*} So, by taking $x$ close enough to $0$ we see that $\beta$ is not $C^{\theta}$ at $0$.

Now notice that 
$$
\beta(x)=\frac{1}{d}\beta(ax)+\frac{1}{d}\alpha(x)
$$
 and $\alpha\in C^{r_0}$. Let $S$ be the set of $x$ such that $\beta$ is not $C^{r_0}$ at $x$. Then, from the above equation, if $ax \in S$ then $x\in S$, \ie $S$ is backward invariant (and non empty) and hence dense.

To show the genericity property we will use the following functional analysis Lemma (it is a concequence of the proof of the open mapping theorem.)

\begin{lemma}
	Let $X$ and $Y$ be Banach spaces and  $L:X\to Y$ be a bounded linear map then either $L$ is onto or the image of $L$ is a first category set.
\end{lemma}

Now consider $X=Y=C^\theta$  and $L(\beta)=d\beta-\beta\circ a$. We have that the image of $L$ is the set of $\alpha\in C^\theta$ such that the corresponding $\beta$ belongs to $C^\theta$. We just showed that $L$ is not surjective, hence the set of $\alpha$ such that the corresponding $\beta$ belongs to $C^\theta$ is a first category set and so its complement is a second category set. 
\end{proof}	

\begin{remark}	
	Notice that if $r_0\in\N$, then $d=a^{r_0}$. And we have that if $\beta\in C^{r_0}$ then we can differentiate the above equation and obtain that $\beta^{(r_0)}$ solves the equation 
	$$
	d\beta^{(r_0)}(x)- a^{r_0} \beta^{(r_0)}(ax)=\alpha^{(r_0)}(x)$$ meaning that $\frac{\alpha^{(r_0)}}{d}$ is cohomologous to $0$, which does not happen for generic $\alpha$.
\end{remark}

 
 With Lemma~6.4 at hand we return to discussion of Example~6.3.
 If we apply Theorem~\ref{rigexpthm} to $L$, $f$ and $r<r_0$ then the fibrations $p_i$ are trivial with point fibers.

 If $r\ge r_0$ we will have $p_1(x,y)=p_2(x,y)=y$, \ie fibrations with circle fiber. Let us show this fact.
 
 Differentiating $\psi_1(x,y)=\psi_2(x+\beta(y),y)$ with respect to $y$ variable yields 
 $$
 \partial_y\psi_1(x,y)=\partial_y\psi_2(x+\beta(y),y) +\partial_x\psi_2(x+\beta(y),y)\beta'(y)
 $$ 
 Notice that $\partial_x\psi_2(x+\beta(y),y)\in C^{r_0-1}$ and, in particular, it is continuous. 
 Let $$U=\{y:\partial_x\psi_2(x+\beta(y),y)\neq 0\;\mbox{for some}\;x\},$$ then $U$ is open and for $y\in U$ and the appropiate $x$, $$\beta'(y)=\frac{ \partial_y\psi_1(x,y)-\partial_y\psi_2(x+\beta(y),y)}{\partial_x\psi_2(x+\beta(y),y)}.$$ So, for $y\in U$ we obtain that that the right hand side is locally in $C^{r_0-1}$ and hence $\beta'$ is locally in $C^{r_0-1}$ for points in $U$. Hence by Lemma \ref{regularitybeta}, $U$ is empty and hence  $\frac{\partial}{\partial x}\psi_2(x+\beta(y),y)=0$ for all $x$ and a dense set of $y\in S^1$. We conclude that $\psi_2$ (and similarly, $\psi_1$) depends solely on the $y$-coordinate.

\end{example}

\begin{example}[Irreducible automorphism of an infratorus]\label{infratorusexample}
	We have explained in Remark~\ref{rem_irr} that (non-trivial) infratori do not support totally irreducible affine automorphisms. Here we show that one can still construct irreducible examples (which become reducible after passing to a finite iterate). 
	
	Define the expanding endomorphism of $\T^3$ by
	$$L=\left(\begin{matrix}
	0&0&3\\
	1&0&0\\
	0&1&0
	\end{matrix}\right)$$ 
	Note that $L^3$ is diagonal.
	Define the holonomy group $\{id, \gamma_1,\gamma_2,\gamma_3\}$ as follows
	$$\gamma_1=\left(\begin{matrix}
	1&0&0\\
	0&-1&0\\
	0&0&-1
	\end{matrix}\right)\;\;\;\;\;\gamma_2=\left(\begin{matrix}
	-1&0&0\\
	0&1&0\\
	0&0&-1
	\end{matrix}\right)\;\;\;\;\;\gamma_3=\left(\begin{matrix}
	-1&0&0\\
	0&-1&0\\
	0&0&1
	\end{matrix}\right)
	$$ 
	Finally let 
	$$v_1=\left(\frac{1}{2},0,\frac{1}{2}\right)^t\;\;\;\;v_2=\left(\frac{1}{2},\frac{1}{2},0\right)^t\;\;\;\;v_3=\left(0,\frac{1}{2},\frac{1}{2}\right)^t
	$$
	and $T_i(x)=\gamma_i(x)+v_i$, $i=1,2,3$. We let $\Gamma$ be the group of affine diffeomorphisms of $\T^3$ generated by the $T_i$'s. It is easy to see that, in fact, $\Gamma=\{Id_{\T^3}, T_1, T_2, T_3\}$ and, hence, $\Gamma$ acts freely on $\T^3$. 
	
	Finally $L\Gamma L^{-1}\subset\Gamma$ and, hence, induces an expanding endomorphism of the infratorus $\T^3/\Gamma$. Indeed, $L\circ T_1\circ L^{-1}=T_2+(1,0,0)^t$, $L\circ T_2\circ L^{-1}=T_3$ and $L\circ T_3\circ L^{-1}=T_1+(1,0,0)^t$.
	
\end{example}

\begin{example}[Seifert fibration]\label{seifertfibration}
	
Recall that in Theorem~\ref{rigexpthm} we assume that manifolds $M_i$ are homeomorphic to nilmanifolds. If $M_i$ are not homeomorphic to nilmanifolds then the construction of compact foliations in the proof of Theorem~\ref{rigexpthm} still works, but these foliations might fail to be fibrations. The example below illustrates this point.

Consider the Klein bottle $\K$ given as a quotient of the torus $\T^2=\R^2/\Z^2$ by the involution $T(x,y)=(x+\frac{1}{2},-y)$. We can also model $\K$ as the rectangle $[0,\frac{1}{2}]\times[-1/2,1/2]$ where the sides are identified by $(x,y)\to(x-1/2,-y)$ and $(x,y)\to(x,y+1)$. One can easily check that the expanding linear map 
$$
L=\left(\begin{matrix} 3&0\\0&2\end{matrix}\right)$$
induces an expanding map $L\colon\K\to\K$. We foliate $\K$ by horizontal curves $\{y=\textup{const}\}$. More precisely, for every $y\in [-1/2,1/2]$ define the circles 
$$
\cC_y=\left\{(t,y):t\in\left[0,\frac{1}{2}\right]\right\}\cup\left\{(t,-y):t\in\left[0,\frac{1}{2}\right]\right\}
$$ 
Notice that if $y\neq 0, \frac{1}{2}$ then $\cC_y$ consists of two segments on the rectangle. For $y=0$, $\cC_0$ is a singular curve that consists of only one segment $[0,1/2]\times\{0\}$ and hence has half of the length of the other leaves. The same happens for $y=\frac{1}{2}$, $\cC_{\frac{1}{2}}$ is a singular curve that consist of only one segment $[0,1/2]\times\{1/2\}\sim [0,1/2]\times\{-1//2\}$. Moreover, notice that $\cC_y=\cC_{-y}$. 

We have defined a foliation on $\K$ which is obviously not a fibration. Indeed, the quotient map $\pi:\K\to S^1/[y\sim -y]$ yields an orbifold structure on $S^1/[y\sim -y]$.

Notice that $L(\cC_y)=\cC_{(2y\mod 1)}$, hence, the foliation is $L$-invariant.

We now define expanding maps $f_i\colon\K \to \K$, $i=1,2$. We let  $f_i(x,y)=(g_i(x),2y)$, where $g_i(x)=3x+\alpha_i(x)$ with $\alpha_i(0)=0$ and $\alpha_i(x+\frac{1}{2})=\alpha_i(x)$ for every $x\in S^1$.  Such formulae define maps on the Klein bottle which are homotopic to $L$. Moreover, these maps are expanding provided that $C^1$ norms of $\alpha_i$ are sufficiently small. Also notice that $f_i$ preserve the foliation $\cC$.


The conjugacy $h$ between $f_1$ and $f_2$, $h\circ f_1=f_2\circ h$ has the form $h(x,y)=(h_0(x),y)$, where $h_0\circ g_1=g_2\circ h_0$.  Notice that by the symmetries of $f_i$, $h_0(x+1/2)=h_0(x)+1/2$ and hence $h$ is indeed the conjugacy on the Klein bottle. We can assume that $\alpha_i$ are chosen so that $h_0$ and, hence, $h$ is not $C^1$.

Take any $\phi_0:\R\to\R$ such that $\phi_0(y+1)=\phi_0(y)$ and $\phi_0(-y)=\phi_0(y)$ \eg $\phi_0(y)=\cos 2\pi y$. Then $\phi(x,y)=\phi_0(y)$ defines a function on $\K$ and $\phi\circ h=\phi$. On the other hand if $\phi_1=\phi_2\circ h$ for some smooth functions $\phi_1$ and $\phi_2$ then both $\phi_1$ and $\phi_2$ must be constant on the leaves of $\cC$ because $h_0$ is non-differentiable on a dense set of $x\in S^1$. So defining $\phi_i=\phi$ for $i=1,2$ we are in the hypothesis of the Theorem~\ref{rigexpthm}. We conclude that $\cC$ is precisely the compact foliation given by the construction in the proof of Theorem~\ref{rigexpthm}.

\end{example}

\begin{example}[Exotic examples]\label{ex_exotic}

Here we explain that the fiber bundle structure given by Theorem~\ref{rigexpthm} could be non-trivial even in the case when the ambient manifold is an exotic torus. Examples of expanding maps on exotic tori were first constructed by Farrell and Jones~\cite{FJ} in dimensions $d \ge 7$. We explain how, with some extra care, the beautiful construction of Farrell-Jones can be adapted to our setting.

Let $\Sigma^d$ be a $d$-dimensional, $d\ge 7$, homotopy sphere and let $\T^d$ be the standard torus. A simple way of constructing an exotic torus is by taking the connected sum $\T^d\#\Sigma$. If $\Sigma^d$ is not homeomorphic to the standard sphere then $\T^d\#\Sigma^d$ is not homeorphic to $\T^d$~\cite[\S15A]{wall}. Further, it is well-known that for $d\ge 7$, one can realize $\T^d\#\Sigma^d$ as $\T^d$ with a disk $\D^d$ removed and then glued back in using an orientation-preserving ``twist diffeomorphism'' $\phi\in\textup{Diff}(\bS^{d-1})$.
$$
\T^d\#\Sigma^d=(\T^d\backslash \D^d)\cup_\phi\D^d
$$
It is easy to check that if $\phi'$ is isotopic to $\phi$ then the corresponding exotic tori are diffeomorphic.

We view the sphere $\bS^{d-1}=\partial\D^d$ as the standard sphere in $\R^d$
$$
\bS^{d-1}=\{(x_1, x_2, \ldots x_d): \sum_ix_i^2=1\}
$$
Cerf~\cite{C} showed that for every homotopy sphere $\Sigma^d$ one can realize $\T^d\#\Sigma^d$ using a diffeomorphism $\phi\colon\bS^{d-1}\to\bS^{d-1}$ which preserves the first coordinate, \ie has the form 
$$
\phi(x_1,x_2, x_3\ldots x_d)=(x_1, x_2', x_3'\ldots x_d')
$$
Then $\phi$ can viewed as a path of diffeomorhisms and gives a representative of an element of $\pi_1(\textup{Diff}(\bS^{d-2}))$. More generally, one can consider the space $\textup{Diff}_k(\bS^{d-1})$ of orientation preserving diffeomorphisms which preserve first $k$ coordinates $x_1, x_2,\ldots x_k$ and, hence, give an element of $\pi_k(\textup{Diff}(\bS^{d-1-k}))$. Isotopy classes of such diffeomorphism form a subgroup $\Gamma_{k+1}^d$ of the group of isotopy classes of all orientation diffeomorphisms $\Theta_d$ (which is identified with the group of homotopy spheres equipped with the connected sum operation). It is known that $\Gamma_{k+1}^d$ is non-trivial in a certain range of pairs $(k,d)$~\cite{ABK2}.

Now we formulate the extra property of $\phi\in\textup{Diff}_k(\bS^{d-1})$ which we will need (and which is not needed in the original Farrell-Jones construction). Consider the obvious homomorphism
$$
\gamma\colon \pi_k(\textup{Diff}(\bS^{d-1-k}))\to\pi_0(\textup{Diff}(\bS^{d-1}))\simeq\Theta_d
$$
\begin{lemma}\textup{(\cite[Proposition 1.2.3; \S1.3]{ABK})}
There exists pairs $(k,d)$\footnote{For specific arithmetic conditions see~\cite[Corollary~1.3.6]{ABK}.} and a torsion element $[\phi]\in \pi_k(\textup{Diff}(\bS^{d-1-k}))$, $[\phi^p]=0$, whose image in $\pi_0(\textup{Diff}(\bS^{d-1}))$ non-trivial, \ie $\gamma[\phi]\neq 0$.
\label{lemmaABK}
\end{lemma}

We proceed to briefly recall the Farrell-Jones construction~\cite{FJ, F} and then explain how the above lemma allows to produce exotic example which admit invariant fibrations with $(d-k)$-dimensional fibers. The construction yields a $\times s$-map on $\pi_1(\T^d\#\Sigma^d)$ for a sufficiently large $s$ which also must satisfy certain congruence arithmetic condition.

We pick a $\phi\in\textup{Diff}_k(\bS^{d-1})$ given by Lemma~\ref{lemmaABK} and realize $\T^d\#\Sigma^d$  by removing a disk $\D^d$ from $\T^d$ and then attaching it back with a twist $\cup_\phi \D^d$. Given an integer $s\ge 2$ consider the manifold $M_s$ which is diffeomorphic to $\T^d\#\Sigma^d$ and which is obtained by removing the conformally scaled disk $\frac1s\D^d$ and then attaching it back with a twist $\cup_\phi \frac1s\D^d$. Because of our choice of $\phi$ both manifolds are naturally total spaces of smooth torus bundles
$$
\T^{d-k}\to\T^d\#\Sigma^d\stackrel{p_1}{\longrightarrow}\T^k,\,\,\,  \T^{d-k}\to M_s\to\T^k
$$
where the base space $\T^k$ corresponds to the first $k$ coordinates fixed by $\phi$.

Let $N\to\T^d\#\Sigma^d$ be the locally isometric cover which induces $\times s$ map on the fundamental group. And let $N_s$ be a copy of $N$ with the Riemannian metric conformally scaled by $\frac1s$. Clearly both $N_s$ and $N$ smoothly fiber over $\T^k$. Then the posited expanding map is the composition
$$
\T^d\#\Sigma^d\stackrel{F_s}{\longrightarrow} M_s\stackrel{G_s}{\longrightarrow} N_s\stackrel{\times s}{\longrightarrow} N\to \T^d\#\Sigma^d
$$
The diffeomorphisms $F_s$ and $G_s$ are constructed with a uniform (in $s$) lower bound on minimal expansion. It immediately follows that for sufficiently large $s$ the composite map $f\colon \T^d\#\Sigma^d\to \T^d\#\Sigma^d$ is uniformly expanding.

The diffeomorphism $F_s$ which ``shrinks'' the exotic sphere is constructed using the ``commutator trick'' and it is easy to check that $F_s$ is fiber preserving and fibers over the identity map $id_{\T^k}$. We claim that the same is true for the diffeomorphism $G_s$. The purpose of $G_s$ is to introduce a certain number of scaled exotic spheres $\cup_\phi\frac1s\D^d$, and, thus, create $N_s$. These exotic spheres are introduced in groups of size $b$ which is divisible by the order of $\phi$ in $\Theta_d$~\cite[Lemma 3]{FJ}. Alternatively one can think of $G_s^{-1}$ as a diffeomorphism which removes exotic spheres in groups of size $b$. To remove one such group one uses diffeomorhism given by the isotopy between $\phi^b$ and $id_{\bS^{d-1}}$. A priori such an isotopy does not preserve the fibers. However, we can require $b$ to be divisible by $p$ which is given by Lemma~\ref{lemmaABK}. Then $\phi^b$ is isotopic to $id_{\bS^{d-1}}$ in the space $\textup{Diff}_k(\bS^{d-1})$ and hence the resulting diffeomorphism $G_s\colon M_s\to N_s$ is fiber-preserving and fibers over $id_{\T^k}$. Finally we notice that the covering map $N\to\T^d\#\Sigma^d$ and the expanding map $\times s\colon N_s\to N$ are fiber preserving as well. We conclude that the expanding map $f$ fibers over $\times s$ map on $\T^k$.
$$
\xymatrix{
\T^d\#\Sigma^d\ar_{p_1}[d]\ar^f[r] & \T^d\#\Sigma^d\ar_{p_1}[d]\\
\T^k\ar^{\times s}[r] & \T^k
}
$$
The same diagram holds for the standard $\times s$ expanding map $E_s\colon \T^d\to \T^d$
$$
\xymatrix{
\T^d\ar_{p_2}[d]\ar^{E_s}[r] & \T^d\ar_{p_2}[d]\\
\T^k\ar^{\times s}[r] & \T^k
}
$$
and it is easy to see that the conjugacy $h\colon \T^d\#\Sigma^d\to\T^d$, $h\circ f= E_s\circ h$, maps fibers to fibers and the induced conjugacy on $\T^k$ is identity, \ie $\bar h=id_{\T^k}$. 

We claim that one can perturb $f$ along the fibers so that the fibrations $p_1$ and $p_2$ are precisely the ones appearing in the Theorem~\ref{rigexpthm}. Indeed consider the restrictions of $f_p\colon \T^{d-k}_p\to\T^{d-k}_p$ and $E_s\colon \T^{d-k}_{h(p)}\to\T^{d-k}_{h(p)}$ to the fibers through the corresponding fixed points. Denote by $p_i'$, $i=1,2$ the fibrations produced by Theorem~\ref{rigexpthm} applied to $f$ and $E_s$. Then the fibers of $p_i'$ refine the fibers of $p_i$ and, hence, we can restrict $p_1'$ and $p_2'$ to $\T^{d-k}_p$ and $\T^{d-k}_{h(p)}$, respectively. Denote by $\ell$ the dimension of the base space for these restricted fibrations. Recall that the induced conjugacy on the base space is smooth. It follows that $\wedge^\ell Df_p$ has $s^\ell$ as an eigenvalue. Hence we perturb $f$ in the neighborhood of $p$ so that $\wedge^\ell Df_p$ does not have $s^\ell$ for an an eigenvalue for all $\ell=1,2,\ldots d-k$. Then we have $\ell=0$ which means that $p_i'=p_i$.

\begin{remark} Similarly, one can perturb $f$ along the fibers to an expanding map $f_2\colon  \T^d\#\Sigma^d\to  \T^d\#\Sigma^d$ such that both $p_i'$ given by Theorem~\ref{rigexpthm} when applied to $f$ and $f_2$ are equal to $p_1$
\end{remark}

\begin{remark} An easier way of constructing an exotic expanding map with non-trivial fibration would be to take the product $f\times L$ of an exotic expanding map $f\colon \T^d\#\Sigma^d\to \T^d\#\Sigma^d$ and a linear expanding map $L\colon\T^m\to \T^m$. Smoothing theory implies that $\T^d\#\Sigma^d\times\T^m$ is not diffeomorphic to $\T^{d+m}$. Then $\T^d\#\Sigma^d\times\T^m$ fibers over $\T^m$ and one can arrange this fibration to be the fibration given by Theorem~\ref{rigexpthm} in a similar way. The example which we described above is more interesting because the smooth structure on $\T^d\#\Sigma^d$ is {\it irreducible}, that is, $\T^d\#\Sigma^d$ is not diffeomorphic to a smooth product of two lower dimensional smooth closed manifolds~\cite[Proposition 1.3]{FG2}.
\end{remark}

\end{example}

\section{Factor version}

We formulate the following generalization of Theorem~\ref{rigexpthm}, where we replace the topological conjugacy by a continuous factor map. The proof follows the same lines with routine modifications and we omit it.

\begin{theorem}
	\label{rigexpthmfactor2} 
	Assume that $M_i$, $i=1,2$, are closed manifolds homeomorphic to a nilmanifold. Let $f_i\colon M_i\to M_i$, $i=1,2$, be $C^r$ smooth, $r\geq 1$, expanding maps and assume that $f_2$ is a topological factor of $f_1$, that is, there exists a continuous map $h:M_1\to M_2$ such that $h\circ f_1=f_2\circ h$.  
	
	Then there exist a $C^r$ expanding map $\bar f:\bar M\to\bar M$  where $\bar M$ is homeomorphic to a nilmanfold, and $C^{r}$ fibrations (with connected fiber homeomorphic to a nilmanifold) $p_i:M_i\to \bar M$, $i=1,2$, such that 
	$$
	p_i\circ f_i=\bar f\circ p_i,\;\;\;\;\; i=1,2
	$$
	Further the map $h$ sends fibers to fibers  
	$$
	p_2\circ h=p_1
	$$
	and the fibrations $p_i$, $i=1,2$, have the following property. If $\phi_i:M_i\to\R$, $i=1,2$, are $C^r$ smooth functions such that for every periodic point  $x\in Fix(f_1^n)$
	$$
	\sum_{k=0}^{n-1}\phi_1(f_1^k(x))=\sum_{k=0}^{n-1}\phi_2(f_2^k(h(x)))
	$$
	then there exist a $C^{r}$ function $\bar\phi:\bar M\to\R$, such that $\phi_i$ is cohomologous to $\bar\phi\circ p_i$ over $f_i$.
\end{theorem}

Using Theorem 7.1 one can naturally study regularity properties of factors maps. We proceed to describe an application.

Let $M_1=N\times M_2$, where $N$ and $M_2$ are nilmanifolds, and let $L\colon M_1\to M_1$ be a product expanding map $L=(A,B)$. Then, clearly, $L$ factors over $B$. Hence if $f_1$ is an expanding map homotopic to $L$ and $f_2$ is an expanding map homotopic to $B$ then $f_1$ factors over $f_2$: $h\circ f_1=f_2\circ h$. 

To define nice invariants of smooth conjugacy we need to introduce a restriction on $L$ and $f_1$. Namely, we assume that the maximal expansion of $A$ is greater than the minimal expansion of $B$. Then the ``vertical foliation'' $N\times \{x\}$, $x\in M_2$, is a weakly expanding foliation. It is easy to see, that for any sufficiently $C^1$ small perturbation $f_1$ of $L$ the weakly expanding foliation survives as an $f_1$-invariant foliation $W^{wu}$.

\begin{corollary} Consider $L, f_1$, $f_2$ are $C^{r+1}$ expanding maps and $h$ is the factor map, $h\circ f_1=f_2\circ h$. Assume that $f_1$ belongs to a sufficiently small $C^1$ neighborhood of $L$. Also assume that $f_2$ is very non-algebraic. If for any periodic point $x=f_1^k(x)$
$$
\frac{\textup{Jac}(f_1^k(x))}{\textup{Jac}(f_1|_{W^{wu}}^k(x))}=\textup{Jac}(f_2^k(h(x))
$$
then the factor map $h$ is $C^r$ smooth.
\end{corollary}
The proof is very similar to the proof of Corollary~\ref{rigexpcor} and we merely provide a sketch. Also one can replace the very non-algebraic assumption on $f_2$ by asking $f_2$ to be an irreducible toral diffeomorphism and assuming that the entropy maximizing measure for $f_2$ is not absolutely continuous.

\begin{proof}[Sketch of the proof] 
Let $p_i\colon M_i\to\bar M_i$ be fibrations given by Theorem~\ref{rigexpthmfactor2} when applied to $f_i$ and $r$. If $\dim \bar M=\dim M_2$ then $p_2$ is a difeomorphism and hence we have that $h=p_2^{-1}\circ p_1$ is $C^r$. 

Hence we need to rule out the possibility that $\dim \bar M<\dim M_2$, \ie the case when the fiber of $p_2$ has dimension $\ge 1$. In this case, following the proof of Corollary~\ref{rigexpcor},  we can apply Theorem~\ref{rigexpthmfactor2} to $\log\textup{Jac}(f_2)$ to conclude that $\log\textup{Jac}(f_2|_{\textup{ker}(p_2)})$ is cohomologous to a function which is constant along the fibers of $p_2$ which yields a contradiction, again, similarly to the proof of Corollary~\ref{rigexpcor}.

One subtle detail, however, is that in order to apply Theorem~\ref{rigexpthmfactor2} one needs to have a pair of $C^r$ functions $(\phi_1,\phi_2)$. We let
$\phi_2=\log\textup{Jac}(f_2)$ and $\phi_1=\log\textup{Jac}(f_1)-\log\textup{Jac}(f_1|_{W^{wu}})$. (Assume for simplicity that $f_i$ are orientation preserving.) It is clear from the assumtion of the corollary that the sums of $\phi_i$ agree along the periodic orbits and it is clear that $\phi_2$ is $C^r$. However, one also need to argue that $\phi_1$ is $C^r$ which is equivalent to $\log\textup{Jac}(f_1|_{W^{wu}})$ being $C^r$.

Smoothness of $\log\textup{Jac}(f_1|_{W^{wu}})$ can be established as follows. Pick a lift $\tilde f_1\colon \tilde M_1\to\tilde M_1$ to the universal cover $\tilde M_1$. Foliation $W^{wu}$ lifts to $\tilde W^{wu}$. Because $\tilde f_1$ is invertible the fast foliation $\tilde W^{uu}$ is also well defined by the standard cone argument. Notice that $\tilde W^{uu}$ is  not equivariant under the group of Deck transformations but this is not going to be important for what follows. Then, by the usual application of the $C^r$ Section Theorem~\cite{HPS}, we have that $\tilde W^{uu}$ is $C^{r+1}$ and hence $\log \textup{Jac}(\tilde f_1|_{\tilde W^{uu}})$ is $C^r$. Finally, extending $\tilde W^{uu}$ to a smooth coordinate system, we have that $Df_1$ has an upper-triangular form and hence 
$$
\log \textup{Jac}(\tilde f_1)=\log \textup{Jac}(\tilde f_1|_{\tilde W^{uu}})+\log \textup{Jac}(\tilde f_1|_{\tilde W^{wu}})
$$
 which implies that $\log \textup{Jac}(\tilde f_1|_{\tilde W^{wu}})$, and hence $\log \textup{Jac}( f_1|_{ W^{wu}})$ is $C^r$.
\end{proof}

\end{document}